\documentclass[12pt,amsfonts]{article}
\usepackage[margin=1in]{geometry}  % set the margins to 1in on all sides
\usepackage{graphicx}              % to include figures
\usepackage[latin1]{inputenc}
\usepackage{amsmath, amssymb, amsthm, epsfig}               % great math stuff
\usepackage{enumerate}
\usepackage{amsfonts}              % for blackboard bold, etc
\usepackage{amsthm}                % better theorem environments
\usepackage{amsthm}
\usepackage{enumerate}
\theoremstyle{plain}
\usepackage{color}
\def\dis{\displaystyle}
\def\nd{\noindent}
\def\thend{\rule{3mm}{3mm}}

\newtheorem{theorem}{Theorem}[section]

\newtheorem{proposition}{Proposition}[section]
\newtheorem{lemma}{Lemma}[section]

\newtheorem{definition}{Definition}[section]
\newtheorem{corollary}{Corollary}[section]
\newtheorem*{theorem*}{Theorem}

\numberwithin{equation}{section}
\begin{document}
\title{Multi-bump positive solutions for a logarithmic Schr\"{o}dinger equation with deepening potential well}
\author{ Claudianor O. Alves\footnote{C.O. Alves was partially supported by CNPq/Brazil  304804/2017-7.} \,\, and \,\, Chao Ji \footnote{C. Ji was partially supported by Natural Science Foundation of Shanghai(20ZR1413900,18ZR1409100).}}

\maketitle

\begin{abstract}
This article concerns the existence of multi-bump positive solutions  for the following logarithmic Schr\"{o}dinger equation
$$
\left\{
\begin{array}{lc}
-\Delta u+ \lambda V(x)u=u \log u^2, & \mbox{in} \quad \mathbb{R}^{N}, \\
u \in H^1(\mathbb{R}^{N}), \\
%u \in H^1(\mathbb{R}^{N}), & \;  \\
\end{array}
\right.
$$
where  $N \geq 1$, $\lambda>0$ is a parameter and the nonnegative continuous function $V: \mathbb{R}^{N}\rightarrow \mathbb{R}$ has a potential well
$\Omega: =\text{int}\, V^{-1}(0)$ which possesses $k$ disjoint bounded components $\Omega=\bigcup_{j=1}^{k}\Omega_{j}$. Using the variational methods, we prove that if the parameter $\lambda>0$
is large enough, then the equation has at least $2^{k}-1$ multi-bump positive solutions.
\end{abstract}



\section{Introduction}

In this article, we are concerned with the existence of multi-bump positive solutions  for the following logarithmic Schr\"{o}dinger equation
$$
\left\{
\begin{array}{lc}
-\Delta u+ \lambda V(x)u=u \log u^2, & \mbox{in} \quad \mathbb{R}^{N}, \\
%u(x)>0, & \mbox{in} \quad \mathbb{R}^{N} \\
u \in H^1(\mathbb{R}^{N}), & \;  \\
\end{array}
\right.
\eqno{(P_{\lambda})}
$$
where $\lambda>0$ is a parameter,  $N \geq 1$ and $V: \mathbb{R}^{N}\rightarrow \mathbb{R}$ is a continuous function satisfying the following conditions:
\begin{itemize}
	\item[\rm ($V1$)] $V(x)\geq 0$, for  all $x\in \mathbb{R}^N$.

	\item[\rm ($V2$)] $\Omega:=\text{int}\, V^{-1}(0)$ is a nonempty bounded open subset with smooth boundary and $\overline{\Omega}=V^{-1}(0)$  where $\text{int}\, V^{-1}(0)$ denotes the set of the interior points of $V^{-1}(0)$.

    \item[\rm ($V3$)] $\Omega$ consists of $k$ components:

    $$
	\Omega=\Omega_{1}\cup\Omega_{2}\cup\cdots\cup \Omega_{k},
	$$
    and $\overline{\Omega}_{i}\bigcap \overline{\Omega}_{j}=\emptyset$ for all $i\neq j$.

\end{itemize}

\begin{definition}
A  solution of the problem  $(P_{\lambda})$ {\color{blue}is} a  function $u \in H^1(\mathbb{R}^{N})$ such that $ u^2\log u^2 \in L^1(\mathbb{R}^{N})$ and
\end{definition}
\begin{equation*}
\displaystyle \int_{\mathbb{R}^N} (\nabla u \nabla v +\lambda V(x)u  v)dx=\displaystyle \int_{\mathbb{R}^N} uv \log u^2dx,\,\, \mbox{for all } v \in C^\infty_0(\mathbb{R}^{N}).
\end{equation*}

In recent years, the logarithmic Schr\"odinger equation
 has   received considerable attention. This class of equations has some important physical applications, such as quantum mechanics, quantum optics, nuclear physics, transport and diffusion phenomena, open quantum systems, effective quantum gravity, theory of superfluidity and Bose-Einstein condensation (see \cite{z} and the references therein). On the other hand, the logarithmic Schr\"odinger equation also raises many difficult mathematical problems, for example, the energy functional associated is not well defined in $H^{1}(\mathbb{R}^N)$, because there exists $u \in H^{1}(\mathbb{R}^N)$ such that $\int_{\mathbb{R}^N}u^{2}\log u^2 \, dx=-\infty$. Indeed, it is enough to consider a smooth function that satisfies
 $$
 u(x)=
 \left\{
 \begin{array}{l}
 (|x|^{N/2}\log(|x|))^{-1}, \quad |x| \geq 3, \\
 0, \quad |x| \leq 2.	
 \end{array}
 \right.
 $$

 In order to overcome this technical difficulty, some authors have used different techniques to study the existence, multiplicity and concentration of the solutions under some assumptions on the potential $V$, which can be seen in  \cite{AlvesdeMorais}, \cite{AlvesdeMoraisFigueiredo}, \cite{AlvesChao}, \cite{AlvesChao1},  \cite{AlvesChao2}, \cite{squassina}, \cite{ASZ}, \cite{DZ}, \cite{cs},  \cite{sz}, \cite{sz2}, \cite{TZ},  \cite{WZ}, \cite{ZZ} and the references therein.

One of the main motivations of this paper goes back to the results for the nonlinear Schr\"odinger equations with deepening potential well of the type
\begin{align}\label{problem1}
\left\{
\begin{aligned}
&-\Delta u+ (\lambda V(x)+Z(x))u= u^p, & \mbox{in} \quad \mathbb{R}^{N}, \\
&u(x)>0, & \mbox{in} \quad \mathbb{R}^{N}, \\
%u \in H^1(\mathbb{R}^{N}), & \;  \\
\end{aligned}
\right.
\end{align}
by supposing that the first eigenvalue of $-\Delta+Z(x)$ on $\Omega_{j}$ under Dirichlet boundary condition is positive for each $j\in \{1, 2, \cdots, k\}$,
$p\in (1, \frac{N+2}{N-2})$ and $N\geq 3$. In \cite{DT}, Ding and Tanaka showed the problem \eqref{problem1} has at least $2^{k}-1$ multi-bump solutions for $\lambda>0$ large enough. These solutions have the following characteristics:

For each non-empty subset $\Gamma\subset \{1, 2, \cdots, k\}$ and $\epsilon>0$ fixed, there exists $\lambda^{*}>0$ such that the problem \eqref{problem1} possesses a solution $u_{\lambda}$, for $\lambda\geq \lambda^{*}=\lambda^{*}(\epsilon)$, satisfying:\\
$$
\Big\vert\int_{\Omega_{j}} \big (|\nabla u_{\lambda}|^2+(\lambda V(x) +Z(x))|u_{\lambda}|^2\big)dx-\Big(\frac{1}{2}-\frac{1}{p+1}\Big)^{-1}c_{j}\Big\vert<\epsilon,\,\, \forall j\in\Gamma
$$
and
$$
\int_{\mathbb{R}^{N}\backslash \Omega_{\Gamma}} \big (|\nabla u_{\lambda}|^2+(\lambda V(x) +Z(x))|u_{\lambda}|^2\big)dx<\epsilon,
$$
where $\Omega_{\Gamma}=\underset{j\in \Gamma}{\bigcup}\Omega_{j}$ and $c_{j}$ is the minimax level of the energy functional related to
the problem
\begin{align*}
 \left\{
\begin{aligned}
&-\Delta u+Z(x)u=u^p,\quad\text{in}\,\, \Omega_{j},\\
&u>0, \quad \quad\quad\quad \quad\quad\quad\,\,\,\text{in}\,\,\Omega_{j},\\
&u=0, \quad \quad\quad\quad \quad\quad\quad\,\,\,\text{on}\,\,\partial \Omega_{j}.
\end{aligned}
\right.
\end{align*}

Later, for the critical growth case, Alves et al. \cite{ADS} considered the existence of multi-bump solutions for the following problem
 $$
\left\{
\begin{array}{lc}
-\Delta u+ (\lambda V(x)+Z(x))u= f(u), & \mbox{in} \quad \mathbb{R}^{N}, \\
u(x)>0, & \mbox{in} \quad \mathbb{R}^{N}, \\
%u \in H^1(\mathbb{R}^{N}), & \;  \\
\end{array}
\right.
$$
where $N\geq 3$. For the case $N=2$ and $f$ having an exponential critical growth, Alves and Souto \cite{AS} obtained the same results. Moreover, these solutions found in \cite{ADS} and \cite{AS}
have the same characteristics of those found in \cite{DT}. For the further research about the nonlinear Schr\"odinger equations with deepening potential well, we refer to \cite{Alves},  \cite{AN}, \cite{ANY}, \cite{BW1}, \cite{BW2}, \cite{GT},  \cite{lz} and their references. \\

In particular, due to our scope, we would like to mention \cite{TZ} where Tanaka and Zhang studied the multi-bump solutions for the spatially periodic logarithmic Schr\"odinger equation
$$
\left\{
\begin{array}{lc}
-\Delta u+  V(x)u=Q(x)u \log u^2, \,\,\, u>0\,\,\,& \mbox{in} \,\,\, \mathbb{R}^{N}, \\
%u(x)>0, & \mbox{in} \quad \mathbb{R}^{N} \\
u \in H^1(\mathbb{R}^{N}), & \;  \\
\end{array}
\right.
\eqno{(LS)}
$$
where $N\geq 1$ and $V(x)$, $Q(x)$ are spatially 1-periodic functions of class $C^{1}$. The authors took
an approach using spatially 2L-periodic problems ($L\gg 1$) and showed the existence of infinitely many multi-bump solutions of equation $(LS)$ which are
distinct under $\mathbb{Z}^{N}$-action. In the present paper, we shall establish the existence of multi-bump solutions by using a different approach from that found in \cite{TZ}.  We also notice that in \cite{AlvesdeMoraisFigueiredo},  Alves et al. have studied problem $(P_{\lambda})$ with $V$ satisfies ($V1$), ($V2$) and
\begin{itemize}
	\item[\rm $(V3)'$] There exists $M_{0}>0$ such that $\vert \{x\in \mathbb{R}^{N}; V(x)\leq M_{0}\}\vert < +\infty$, where $\vert A\vert$ denotes the Lebesgue
measure of a measurable set $A\subset \mathbb{R}^{N} $.
\end{itemize}
On one hand, by using condition $(V3)'$, it is easy to overcome the difficulty of lack of  compactness in the whole space $\mathbb{R}^{N}$.
On the other hand, the authors in \cite{AlvesdeMoraisFigueiredo} cannot obtain the multi-bump solutions. Recently, Alves and Ji \cite{AlvesChao2} used the variational methods to prove the existence and concentration of  positive solutions  for a logarithmic Schr\"{o}dinger equation under a local assumption on the potential $V$. In that paper, in order to prove (PS) condition, we modified the nonlinearity in a special way to work an auxiliary problem. By making some new estimates, we proved that the solutions obtained for the auxiliary  problem are solutions  of the original problem when $\epsilon>0$ is sufficient small. Moreover, since the functional associated with the auxiliary  problem  lost some other good properties, we developed a new method to prove the  boundedness of (PS) sequence. Inspired by \cite{AlvesChao2, Alves, DT},  the main purpose
of the present paper is to investigate the existence and multiplicity of multi-bump  positive solutions, as in  \cite{DT}, for the problem $(P_\lambda)$ by  adapting the penalization method found in del Pino and Felmer \cite{DF}.\\

The main result to be proved is the theorem below.

\begin{theorem}\label{teorema}
Suppose that $V$ satisfies $(V1)$--($V3$).  Then, for any non-empty subset $\Gamma$
of $\{1, 2, \cdots, k\}$, there exists $\lambda^{*}>0$ such that, for all $\lambda\geq \lambda^{*}$, the problem $(P_{\lambda})$ has a positive solution $u_{\lambda}$. Moreover, the family $\{u_{\lambda}\}_{\lambda\geq \lambda^{*}}$ has the following properties:  for any sequence $\lambda_{n}\rightarrow \infty$, we can extract a subsequence $\lambda_{n_{i}}$ such that $u_{\lambda_{n_{i}}}$
converges strongly in  $H^{1}(\mathbb{R}^N)$ to a function $u$ which satisfies $u(x)=0$ for $x\not\in \Omega_{\Gamma}$ and the
restriction $u|_{\Omega_{j}}$ is a least energy solution of
\begin{align*}
 \left\{
\begin{aligned}
&-\Delta u=u \log u^2,\quad\,\text{in} \,\,\Omega_{\Gamma},\\
&u>0, \quad \quad\quad\quad \quad\,\,\,\,\text{in}\,\, \Omega_{\Gamma},\\
&u=0, \quad \quad\quad\quad \quad\,\,\,\, \text{on}\,\,\partial \Omega_{\Gamma},
\end{aligned}
\right.
\end{align*}
where $\Omega_{\Gamma}=\underset{j\in \Gamma}{\bigcup}\Omega_{j}$.
\end{theorem}

\begin{corollary}\label{cor}
Under the assumptions of Theorem $1.1$, there exists $\lambda_{*}>0$
such that, for all $\lambda\geq \lambda_{*}$, the problem $(P_{\lambda})$ has at least $2^{k}-1$ positive solutions.
\end{corollary}

In the above mentioned papers \cite{AlvesChao2}, \cite{Alves} and \cite{DT}, it was essential the method developed in \cite{DF},
which consists in modifying the nonlinearity to obtain an auxiliary problem, whose energy functional associated satisfies the (PS)
condition. After that, making some estimates, it is possible to prove that the solutions obtained for the auxiliary problem are in fact
solutions for the original problem when $\lambda$ is large enough. However, since the natural energy functional associated with the problem $(P_\lambda)$ given by
$$
I_{\lambda}(u)=\frac{1}{2}\int_{\mathbb{R}^N}\big (|\nabla u|^2+(\lambda V(x)+1)|u|^2\big)dx-\frac{1}{2}\int_{\mathbb{R}^N}u^2\log u^{2} \,dx,
$$
is not well defined in $H^{1}(\mathbb{R}^N)$, because there exists a function $u$ in that space such that $I_{\lambda}(u)=+\infty$, we cannot directly use the critical points theory for $C^{1}$ functional, then we need to use a different approach from that used  to  the Schr\"{o}dinger equation \eqref{problem1}. Here, we have used strongly the fact that the functional $I_{\lambda}$  is of class $C^{1}$ in  $H^{1}(\mathcal{D})$, when $\mathcal{D} \subset \mathbb{R}^N$ is a bounded domain. Based on this observation, for each $R>0$ and $\lambda>0$ large, we first find a solution $u_{\lambda, R}\in H_{0}^{1}(B_{R}(0))$, and after, taking the limit of $R\rightarrow +\infty$, we get a solution for the original problem. On the other hand, for the nonlinear term $u^{p}$ in \eqref{problem1}, it is easy to verify that $\lim_{t\rightarrow 0}t^{p}/\vert t\vert=0$ as $t\rightarrow 0$
and the function $t^{p}/ t$ is increasing for $t\in (0, +\infty)$ which are very important to use the method in \cite{DF}. But for our problem, the nonlinear term is $u\log u^{2}+u$, it is clear that $t\log t^{2}+t\neq o(t)$ as $t\rightarrow 0$. Thus, we cannot apply directly  del Pino and Felmer's method in \cite{DF} and our problem is more difficult and complicated. The plan of the paper is as follows: In Section 2 we prove the existence of multi-bump solutions for an auxiliary problem in the ball $B_R(0)$ for $R>0$ , while in Section 3 we prove Theorem \ref{teorema}.

\vspace{0.5 cm}

\noindent \textbf{Notation:} From now on in this paper, otherwise mentioned, we use the following notations:
\begin{itemize}
	\item $B_r(u)$ is an open ball centered at $u$ with radius $r>0$, $B_r=B_r(0)$.
	
	\item If $g$ is a mensurable function, the integral $\dis\int_{\mathbb{R}^N}g(x)\,dx$ will be denoted by $\dis\int g(x)\,dx$.
	
	\item   $C$ denotes any positive constant, whose value is not relevant.
	
	\item  $|\,\,\,|_p$ denotes the usual norm of the Lebesgue space $L^{p}(\mathbb{R}^N)$, for $p \in [1,+\infty]$,
    $\Vert\,\,\,\Vert$ denotes the usual norm of the Sobolev space $H^{1}(\mathbb{R}^N)$.

    \item  $H_c^{1}(\mathbb{R}^N)=\{u \in H^{1}(\mathbb{R}^N)\,:\, u \,\, \mbox{has compact support}\, \}.$
	
	\item $o_{n}(1)$ denotes a real sequence with $o_{n}(1)\to 0$ as $n \to +\infty$.
	
	\item $2^*=\frac{2N}{N-2}$ if $N \geq 3$ and $2^*=+\infty$ if $N=1,2$.
\end{itemize}

\section{An auxiliary problem on the ball $B_{R}(0)$}

We shall work on the following space of functions:
$$
E_{\lambda}=\Big\{u\in H^{1}(\mathbb{R}^N): \int V(x)\vert u\vert^{2} dx<+\infty\Big\},
$$
endowed with the norm
$$
\Vert u\Vert_{\lambda}=\Big(\int \big (|\nabla u|^2+({\lambda V(x)} +1)|u|^2\big)dx\Big)^{\frac{1}{2}}.
$$
Since $V(x)\geq 0$ for all $x\in \mathbb{R}^N$, the embedding $E_{\lambda}\hookrightarrow H^{1}(\mathbb{R}^N)$ is continuous,
and so, the embedding $E_{\lambda}\hookrightarrow L^{q}(\mathbb{R}^N)$ is also continuous for all $q\in [2, 2^{*}]$.

For each $R>0$, we define a norm $\| \cdot\|_{\lambda,R}$  on $H^{1}_{0}(B_R(0))$ by
$$
\|u\|_{\lambda,R}=\left(\int_{B_R(0)}(|\nabla u|^{2}+(\lambda V(x)+1)|u|^{2})\,dx\right)^{\frac{1}{2}}
$$
which is equivalent to the usual norm in that space for all $\lambda,R>0$. In what follows we will denote by $E_{\lambda,R}$ the space $H^{1}_{0}(B_R(0))$  endowed with the norm $\|\;\;\;\|_{\lambda,R}$.

Following the approach explored in \cite{cs,sz}, for a small  $\delta>0$, let us define the following functions:

$$
F_1(s)=
\left\{ \begin{array}{lc}
0, & \; s=0 \\
-\frac{1}{2}s^2\log s^2, & \; 0<|s|<\delta \\
-\frac{1}{2}s^2(
\log\delta^2+3)+2\delta|s|-\frac{1}{2}\delta^2,  & \; |s| \geq \delta
\end{array} \right.
$$

\noindent and

$$
F_2(s)=
\left\{ \begin{array}{lc}
0, & \;  |s|<\delta \\
\frac{1}{2}s^2\log(s^2/\delta^2)+2\delta|s|-\frac{3}{2}s^2-\frac{1}{2}\delta^2,  & \; |s| \geq  \delta.
\end{array} \right
.$$
Therefore
\begin{equation}\label{efes}
F_2(s)-F_1(s)=\frac{1}{2}s^2\log s^2, \quad \forall s \in \mathbb{R}.
\end{equation}

It was proved in \cite{cs} and \cite{sz} that $F_1$ and $F_2$ verify the following properties:
\begin{equation*}\label{eq1}
F_1, F_2 \in C^1(\mathbb{R},\mathbb{R}).
\end{equation*}
If $\delta >0$ is small enough, $F_1$ is convex, even, $F_1(s)\geq 0$ for all $ s\in \mathbb{R}$ and
\begin{equation*}\label{eq2}
F'_1(s)s\geq 0, \ s \in \mathbb{R}.
\end{equation*}
For each fixed $p \in (2, 2^*)$, there exists $C>0$ such that
\begin{equation}\label{eq5}
|F'_2(s)|\leq C|s|^{p-1}, \quad \forall s \in \mathbb{R}.
\end{equation}

By a simple observation, it is easy to see that
\begin{equation*}
	\frac{F'_2(s)}{s}\,\,\, \text{is nondecreasing for}\,\, s>0\quad\text{and}\quad\frac{F'_2(s)}{s}\,\, \text{is strictly increasing for}\,\, s>\delta,
\end{equation*}
\begin{equation*}
	\lim_{s\rightarrow +\infty}\frac{F'_2(s)}{s}=+\infty,
\end{equation*}
and
\begin{equation*}
	F'_2(s)\geq 0 \,\, \text{for}\,\, s>0\,\, \text{and}\,\,F'_2(s)> 0 \,\, \text{for}\,\, s>\delta.
\end{equation*}

\subsection{The auxiliary problem}

For each $j\in \{1, \cdots, k\}$,  we fix a bounded open subset $\Omega'_{j}$ with smooth boundary such that
(i)\,$ \overline{\Omega_{j}}\subset \Omega'_{j}$, \\
(ii)\,$\overline{\Omega'_{j}}\cap \overline{\Omega'_{l}}=\emptyset, \,\, \text{for all}\,\, j\neq l$.\\
From now on, we fix a non-empty subset $\Gamma\subset \{1, \cdots, k\}$, $R>0$ such that  $\Omega'_{\Gamma} \subset B_R(0)$ and
$$
\Omega_{\Gamma}=\bigcup_{j\in \Gamma}\Omega_{j},\quad \Omega'_{\Gamma}=\bigcup_{j\in \Gamma}\Omega'_{j}.
$$

Let $l>0$ small,  $a_{0}>0$ such that $\frac{F'_2(a_{0})}{a_{0}}=l$, it is clear that $a_{0}>\delta$. We define
$$
\tilde{F}'_2(s)=
\left\{ \begin{array}{lc}
	F'_2(s),& \; 0\leq s\leq a_{0} \\
	ls,  & \; s \geq a_{0}.
\end{array} \right.
$$
and
\begin{equation*}
G'_2(x, t)=\chi_{\Gamma}(x)F'_2(t)+(1-\chi_{\Gamma}(x))\tilde{F}'_2(t),
\end{equation*}
where
$$
\chi_{\Gamma}(x)=\left\{ \begin{array}{lc}
1 & \; x\in \Omega'_{\Gamma},\\
0 & \;  x\in B_R(0)\backslash \Omega'_{\Gamma}.
\end{array} \right.
$$
Now we consider the existence of solution for the problem
$$
\left\{
\begin{array}{l}
-\Delta u+ (\lambda V(x)+1)u=G'_2(x, u^{+})- F'_1(u),\,\, \text{in} \,\, B_R(0), \\
\, u=0, \quad \quad\quad \quad\quad \quad\quad \quad\quad \quad\quad \quad\quad \quad\quad \quad\,\,\text{on} \,\, \partial B_R(0).
\end{array}
\right.
\eqno{(M_{\lambda, R})}
$$
Notice that, if $u_{\lambda, R}$ is a positive solution of $(M_{\lambda, R})$ with $0< u_{\lambda, R}\leq a_{0}$ for all $x\in B_R(0)\backslash \Omega'_{\Gamma}$, then $G'_2(x, u_{\lambda, R})=F'_2(u_{\lambda, R})$, and so, $u_{\lambda, R}$ is also a positive solution of

$$
\left\{
\begin{array}{l}
-\Delta u+\lambda V(x)u=u \log u^2,\quad\quad\quad \text{in} \,\, B_R(0),\\
u>0, \quad\quad \quad\quad \quad\quad \quad\quad\quad \quad\quad \,\,\,\,\text{in} \,\, B_R(0),\\
u=0, \quad\quad \quad\quad \quad\quad \quad\quad\quad \quad\quad\,\,\,\,\text{on} \,\, \partial B_R(0).
\end{array}
\right.
\eqno{(P_{\lambda, R})}
$$

 Moreover, we shall look for the nontrivial critical points for the functional
\begin{equation*}
{\Phi}_{\lambda, R}(u)=\frac{1}{2}\int_{ B_R(0)}\big (|\nabla u|^2+(\lambda V(x)+1)|u|^2\big)dx+\int_{ B_R(0)} F_{1}(u)dx-\int_{B_R(0)}  G_{2}(x, u^{+})dx,
\end{equation*}
 where
$$
u^{+}=\max\{u(x), 0\} \quad \mbox{and} \quad G_{2}(x, t)=\int_{0}^{t}G'_{2}(x, s)ds \quad \forall\, (x, t) \in B_R(0)\times \mathbb{R}.
$$
It is standard to see that $\Phi_{\lambda, R}\in C^{1}(E_{\lambda,R}, \mathbb{R})$.\\

Our first lemma establishes that the functional ${\Phi}_{\lambda, R}$ satisfies the mountain pass geometry \cite{Willem}.

\begin{lemma}\label{vicente} For all $\lambda> 0$, the functional $\Phi_{\lambda, R}$ satisfies the following conditions:\\
\noindent (i)  $\Phi_{\lambda, R}(0)=0$;\\
\noindent (ii) there exist $\alpha, \rho>0$ such that ${\Phi}_{\lambda, R}(u)\geq \alpha$ for any $u\in E_{\lambda,R}$ with $\|u\|_{\lambda,R}=\rho$;\\
\noindent (iii) there exists $e\in E_{\lambda,R}$ with $\Vert e\Vert_{\lambda, R}>\rho$ such that ${\Phi}_{\lambda, R}(e)<0$.
\end{lemma}

\begin{proof} \mbox{} \\
\noindent $(i)$: It is clear.\\
\noindent $(ii)$:
 Note that ${\Phi}_{\lambda, R}(u)\geq\dis\frac{1}{2}\|u\|_{\lambda, R}^2-\int_{B_R(0)} F_{2}(u^+)dx$. Hence,  from (\ref{eq5}),
$$
{\Phi}_{\lambda, R}(u)\geq \dis\frac{1}{2}\|u\|_{\lambda, R}^2-C\|u\|_{\lambda, R}^p\geq C_1> 0,
$$
for some $C_1>0$ and $\|u\|_{\lambda, R}>0$ small enough. Here the constant $C_1$ does not depend on $\lambda$ and $R$. \\

\noindent $(iii)$:  Fixing  $0<v\in C_{0}^{\infty}(\Omega_{\Gamma})$, by \eqref{efes},
$$
{\Phi}_{\lambda,R}(sv)= s^2 \left[{\Phi}_{\lambda,R}(v)-\log s \dis\int_{B_{R}(0)} v^2dx\right]\rightarrow - \infty, \, \mbox{as} \, s \rightarrow +\infty.
$$
Thereby, there exists $s_0>0$ independent of $\lambda>0$ and $R>0$ large such that $\Phi_{\lambda,R}(sv)<0$.
\end{proof}

The mountain pass level associated with ${\Phi}_{\lambda,R}$, denoted by $c_{\lambda, R}$, is given by
$$
c_{\lambda, R}=\inf_{\gamma\in \Gamma_{\lambda, R}}\max_{t\in [0, 1]}{\Phi}_{\lambda,R}(\gamma(t)),
$$
where
$$
\Gamma_{\lambda, R}=\{\gamma\in C([0, 1], E_{\lambda,R}): \gamma(0)=0\,\, \text{and}\,\, {\Phi}_{\lambda,R}(\gamma(1))<0\}.
$$
Notice that by Lemma \ref{vicente},
$$
c_{\lambda, R}\geq \alpha>0, \quad \forall\, \lambda>0,\,\, R>0\,\, \text{large}.
$$
In order to show the boundedness of $(PS)$ sequence of $\Phi_{\lambda,R}$, we need a new logarithmic inequality, whose proof can be found in del Pino and Dolbeault \cite[pg 153]{dJ}.

\begin{lemma}  \label{P1}  There exist constants $A,B>0$ such that
	$$
	\int |u|^2\log(|u|^2)\, dx \leq A+ B\log(\|u\|), \quad \forall  u \in H^{1}(\mathbb{R}^N) \setminus \{0\}.
	$$
\end{lemma}

As an immediate consequence we have the corollary:

\begin{corollary}  \label{C1}  There exist $C,\Theta>0$ such that if $u \in H^{1}(\mathbb{R}^N)$ and $\|u\| \geq \Theta$, then
	$$
	\int \log(|u|^2)|u|^2\,dx \leq C (1+ \|u\|).
	$$
\end{corollary}

By the definition of $G_2$, it is easy to see that
$$
G_2(x,s) \leq F_2(s), \quad s \geq 0.
$$
Consequently
\begin{align}\label{ine}
{\Phi}_{\lambda,R}(u) \geq  \displaystyle \frac{1}{2}\int_{ B_R(0)}\big (|\nabla u|^2+(\lambda V(x)+1)|u|^2\big)dx\displaystyle -\frac{1}{2}\int_{ B_R(0)}(u^{+})^{2} \log(u^{+})^{2},
\end{align}
for all $ u \in E_{\lambda,R}$.

\begin{lemma}\label{boundedness}
 Let $(u_n) \subset E_{\lambda,R}$ be a sequence such that $(\Phi_{\lambda,R}(u_n))$ is bounded in $\mathbb{R}$. Then, $(u_n)$ is a bounded sequence in $E_{\lambda,R}$.
\end{lemma}	
\begin{proof}	 Since  $(\Phi_{\lambda,R}(u_n))$ is bounded, there exists $M>0$ such that
$$
M \geq \Phi_{\lambda,R}(u_n), \quad \forall n \in \mathbb{N}.
$$
Thus, by \eqref{ine},
$$
M \geq   \displaystyle \frac{1}{2}\int_{ B_R(0)}\big (|\nabla u_n|^2+(\lambda V(x)+1)|u_n|^2\big)dx\displaystyle -\frac{1}{2}\int_{ B_R(0)}(u_n^{+})^{2} \log(u_n^{+})^{2}, \\
$$
that is,
\begin{equation*}
M \geq \displaystyle \frac{1}{2}\|u_n\|_{\lambda, R}^{2}-\frac{1}{2}\int_{ B_R(0)} |u^+_n|^{2}\log(|u^+_n|^{2})\,dx,
\end{equation*}
from where it follows that
\begin{equation} \label{Z1}
\|u_n\|_{\lambda, R}^{2} \leq 2M+\int_{ B_R(0)} |u^+_n|^{2}\log(|u^+_n|^{2})\,dx, \quad \forall n \in \mathbb{N}.
\end{equation}
Without loss of generality we will assume that $u^+_n \not=0$, because otherwise, we have the inequality
$$
\|u_n\|_{\lambda, R}^{2} \leq 2M.
$$
From this, assume that there exists $n \in \mathbb{N}$ such that $\|u^+_n\|_{\lambda, R} \geq \Theta$. By Corollary \ref{C1} and \eqref{Z1},
$$
\|u_n\|_{\lambda, R}^{2} \leq 2M + C(1+\|u^+_n\|_{\lambda, R}) \leq 2M + C(1+\|u_n\|_{\lambda, R}).
$$
If $0<\|u^+_n\|_{\lambda, R} \leq \Theta$, Lemma \ref{P1} combined with \eqref{Z1} gives
$$
\|u_n\|_{\lambda, R}^{2}\leq 2M+A +B\log(\Theta).
$$
The above analysis ensures that $(u_n)$ is bounded in $E_{\lambda,R}$.
\end{proof}

As a byproduct of the last lemma we have the following result.
\begin{corollary} \label{coro} If $u_n \in E_{\lambda_n,R_n}$ with $\lambda_n,R_n \to +\infty$ and $(\Phi_{\lambda_n,R_n}(u_n))$ is bounded in $\mathbb{R}$, then $(\|u_{\lambda_n,R_n}\|_{\lambda_n,R_n})$ is also bounded.
	
\end{corollary}

By Lemma \ref{boundedness}, it is also easy to obtain the boundedness of $(PS)$ sequences for $\Phi_{\lambda,R}$.

\begin{corollary}\label{coro0}
 If $(u_n)$ is a $(PS)$ sequence for $\Phi_{\lambda,R}$, then $(u_n)$ is bounded in $E_{\lambda,R}$.
	
\end{corollary}

Our next lemma shows that $\Phi_{\lambda,R}$ verifies the $(PS)$ condition.

\begin{lemma} \label{compacidadeR} The functional $\Phi_{\lambda,R}$ satisfies the $(PS)$ condition.
	
\end{lemma}

\begin{proof} Let $(u_n) \subset E_{\lambda,R}$ be a $(PS)_d$ sequence for $\Phi_{\lambda,R}$, that is,
$$
\Phi_{\lambda,R}(u_n) \to d \quad \mbox{and} \quad \Phi'_{\lambda,R}(u_n) \to 0.
$$	
By Corollary \ref{coro0}, the sequence $(u_n)$ is bounded in $E_{\lambda,R}$, then without loss of generality
 we can assume that there exist $u \in E_{\lambda,R}$, and a subsequence of $(u_n)$, still denoted by itself, such that
$$
u_n \rightharpoonup u, \quad \mbox{in}\,\,  E_{\lambda,R},
$$
$$
u_n \to u \quad \mbox{in} \quad L^{q}(B_R(0)), \quad \forall\, q \in [1,2^*),
$$
and
$$
u_n(x) \to u(x) \quad \mbox{a.e. in} \quad B_R(0).
$$
For all $t \in \mathbb{R}$ and fixing $p \in (2,2^*)$, there exists $C>0$ such that
$$
|G'_{2}(x, t)| \leq l\vert t\vert+C|t|^{p-1}, \quad \forall t \in \mathbb{R},
$$	
and
$$
|F'_{1}(t)| \leq C(1+|t|^{p}), \quad \forall t \in \mathbb{R}.
$$	
Hence, by the Sobolev embeddings,
$$
\int_{B_R(0)}G'_{2}(x, u_n^{+})u_n^{+}\,dx \to \int_{B_R(0)}G'_{2}(x, u^{+})u^{+}\,dx, \quad \int_{B_R(0)}F'_{1}(u_n)u_n \,dx \to \int_{B_R(0)}F'_{1}(u)u\,dx,
$$
and
$$
\int_{B_R(0)}G'_{2}(x, u_n^{+})v\,dx \to \int_{B_R(0)}G'_{2}(x, u^{+})v\,dx, \quad \int_{B_R(0)}F'_{1}(u_n)v\,dx \to \int_{B_R(0)}F'_{1}(u)v\,dx,
$$
for any $ v \in E_{\lambda,R}$.\\

Now, using the limits $\Phi'_{\lambda,R}(u_n)u_n=\Phi'_{\lambda,R}(u_n)u=o_{n}(1)$, we get
\begin{eqnarray*}
\|u_n-u\|^{2}_{\lambda, R}&=&\int_{B_R(0)}\Big(G'_{2}(x, u_n^{+})- G'_{2}(x, u^{+})\Big)(u_n^{+}-u^{+})\,dx\\
&&-\int_{B_R(0)}\Big(F'_{1}(u_n)-F'_{1}(u)\Big)(u_n-u)\,dx+o_{n}(1)=o_{n}(1),
\end{eqnarray*}
showing the desired result. 	
\end{proof}

\vspace{0.5 cm}

\begin{theorem}\label{teorema1}
The problem  $(M_{\lambda, R})$  has a positive solution $u_{\lambda, R}\in E_{\lambda,R}$ such that $\Phi_{\lambda,R}(u_{\lambda, R})=c_{\lambda, R}$,
where $c_{\lambda, R}$ denotes the mountain pass level associated with $\Phi_{\lambda,R}$.
\end{theorem}
\begin{proof}
The existence of the nontrivial solution $u_{\lambda, R}$ is an immediate result of Lemma \ref{vicente}, Corollary \ref{coro0} and Lemma \ref{compacidadeR}.
The function $u_{\lambda, R}$ is nonnegative, because
$$
\Phi_{\lambda,R}(u_{\lambda, R})(u_{\lambda, R}^{-})=0\Rightarrow u_{\lambda, R}^{-}=0,
$$
where $u_{\lambda, R}^{-}=\min \{u_{\lambda, R}, 0 \}$.  By the maximum principle(see \cite[Theorem 1]{V}), we have
$u_{\lambda, R}(x)>0$ for a.e. $x\in B_R(0)$.
\end{proof}

\subsection{The $(PS)_{\infty, R}$ condition}\label{tpm}
In the sequel, for each $R>0$, we study the behavior of a $(PS)_{\infty, R}$ sequence for $\Phi_{\lambda,R}$, that is, a sequence $(u_{n}) \subset H^{1}_0(B_R(0))$ satisfying
$$\aligned
& u_{n}\in E_{\lambda_n,R}\,\, \text{and}\,\, \lambda_{n}\rightarrow\infty,\\
&\Phi_{\lambda_{n},R}(u_n)\rightarrow c,\\
&\Vert \Phi'_{\lambda_{n},R}(u_n)\Vert \rightarrow 0.
\endaligned \eqno$$

\begin{proposition}\label{prop1}
Let $(u_{n}) \subset H^{1}_0(B_R(0))$ be a $(PS)_{\infty, R}$ sequence. Then, for some subsequence, still denoted by $(u_{n})$,
there exists $u\in H^{1}_0(\Omega)$ such that
$$
u_{n}\rightharpoonup u\quad\text{in}\,\, H^{1}_0(B_R(0)).
$$
Moreover,\\
\noindent (i) $u_{n}$ converges to $u$ in the strong sense, that is,
$$
\Vert u_{n}-u\Vert_{\lambda_{n}, R}\rightarrow 0.
$$
Hence
$$
u_{n}\rightarrow u\quad \text{in}\,\, H^{1}_0(B_R(0)).
$$
\noindent (ii)  $u\equiv0$ in $B_R(0)\backslash \Omega_{\Gamma}$ and $u$ is a  solution of
$$
\left\{
\begin{array}{lc}
-\Delta u=u \log u^2, & \mbox{in} \quad \Omega_{\Gamma}, \\
u=0,  & \mbox{on} \quad \partial \Omega_{\Gamma}.  \\
\end{array}
\right.
\eqno{(P_{\infty, \Gamma})}
$$

\noindent (iii) $u_{n}$ also satisfies
$$\aligned
& \lambda_{n}\int_{B_R(0)} V(x)|u_{n}|^2dx \rightarrow 0,\\
&\Vert u_{n}\Vert^{2}_{\lambda_{n},B_R(0)\backslash \Omega_{\Gamma}}\rightarrow 0,\\
&\Vert u_{n}\Vert^{2}_{\lambda_{n}, \Omega'_{j}}\rightarrow \int_{\Omega_{j}}\big (|\nabla u|^2+|u|^2\big)dx,\,\,\text{for all}\,\, j\in \Gamma.
\endaligned \eqno$$

\end{proposition}

\begin{proof}
By Corollary \ref{coro}, there exists $K>0$ such that
$$
\Vert u_{n}\Vert^{2}_{\lambda_{n}, R}\leq K, \quad \forall n\in \mathbb{N}.
$$
Thus $(u_{n})$ is bounded in $H^{1}_0(B_R(0))$ and we can assume that for some $u\in H^{1}_0(B_R(0))$
$$
u_{n}\rightharpoonup u\quad\text{in}\,\, H^{1}_0(B_R(0))
$$
and
$$
u_{n}(x)\rightarrow u(x)\quad\text{a.e. in}\,\, B_R(0).
$$

Fixing  $C_{m}=\{x\in B_R(0): V(x)\geq \frac{1}{m}\}$, we have
$$
\int_{C_{m}} |u_{n}|^2dx \leq \frac{m}{\lambda_{n}}\int_{B_R(0)} \lambda_{n} V(x)|u_{n}|^2dx,
$$
that is
$$
\int_{C_{m}} |u_{n}|^2dx \leq \frac{m}{\lambda_{n}}\|u_{n}\|_{\lambda_{n}, R}^2.
$$
The last inequality together with the Fatou's Lemma yields
$$
\int_{C_{m}} |u|^2dx=0, \quad \forall m\in \mathbb{N}.
$$
Then $u(x)=0$ on $\bigcup_{m=1}^{+\infty} C_{m}=B_R(0)\backslash \overline{\Omega}$, and so,
$u|_{\Omega_{j}}\in H_{0}^{1}(\Omega_{j})$, $j\in\{1, \cdots, k\}$. From this, we are able to prove $(i)-(iv)$\\
$(i)$. Since $u=0$ in $B_R(0)\backslash \overline{\Omega}$  and $\Phi'_{\lambda_{n},R}(u_n)u_n=\Phi'_{\lambda_{n},R}(u_n)u=o_n(1)$, we have
$$
\int_{B_R(0)}(\vert \nabla u_{n}-\nabla u\vert^{2}+(\lambda_{n}V(x)+1)\vert  u_{n}- u\vert^{2})dx\rightarrow 0
$$
which implies that $u_{n}\rightarrow u$ in $H^{1}_0(B_R(0))$.\\
$(ii)$ Since $u\in H^{1}_{0}(B_R(0))$ and $u=0$ in $B_R(0)\backslash \overline{\Omega}$, we deduce that $u\in H_{0}^{1}(\Omega)$, or equivalently
$u|_{\Omega_{j}}\in H_{0}^{1}(\Omega_{j})$, for $j=1, \cdots, k$. Moreover, $u_{n}\rightarrow u$ in $H^{1}_0(B_R(0))$ combined with
 $\Phi'_{\lambda_{n}, R}(u_{n})\varphi\rightarrow 0$ as $n\rightarrow +\infty$ for each $\varphi\in C_{0}^{\infty}(\Omega_{\Gamma})$ implies that
\begin{equation*}
\int_{\Omega_{\Gamma}}\big (\nabla u \nabla \varphi+u \varphi\big)dx+\int_{ \Omega_{\Gamma}} F'_{1}(u)\varphi dx-\int_{ \Omega_{\Gamma}}  F'_{2}(u^{+})\varphi dx=0,
\end{equation*}
from where it follows that $u|_{\Omega_{\Gamma}}$ is a  solution for $(P_{\infty, \Gamma})$.

On the other hand, for each $j\in\{1, 2, \cdots, k\}\backslash \Gamma$,
\begin{equation*}
\int_{\Omega_{j}}\big (\vert \nabla u \vert^{2}+u^{2}\big)dx+\int_{ \Omega_{j}} F'_{1}(u)u dx-\int_{ \Omega_{j}}  \tilde{F}'_2(u^{+}) u^{+}dx=0.
\end{equation*}
By the fact that $F'_{1}(s)s\geq 0$ and  $\tilde{F}'_2(s)s\leq ls^{2}$ for all $s\in \mathbb{R}^{+}$, we derive that
\begin{equation*}
\int_{\Omega_{j}}\big (\vert \nabla u\vert^{2}+u^{2}\big)dx\leq \int_{ \Omega_{j}}  \tilde{F}'_2(u^{+}) u^{+} dx\leq l \int_{\Omega_{j}} u^{2}dx.
\end{equation*}
Since $l<1$, $u=0$ in $\Omega_{j}$ for $j\in\{1, 2, \cdots, k\}\backslash \Gamma$ and $u\geq 0$ in $B_R(0)$, which shows (ii).\\
To prove (iii), note that, from (i)
$$
\int_{B_R(0)} \lambda_{n} V(x)|u_{n}|^2dx=\int_{B_R(0)} \lambda_{n} V(x)|u_{n}-u|^2dx\leq C\Vert u_{n}-u\Vert^{2}_{\lambda_{n}, R},
$$
so
$$
\int_{B_R(0)} \lambda_{n} V(x)|u_{n}|^2dx\rightarrow 0\quad \text{as}\,\,n\rightarrow +\infty.
$$
Moreover, from (i) and (ii), it is easy to check that
$$\aligned
&\Vert u_{n}\Vert^{2}_{\lambda_{n},B_R(0)\backslash \Omega_{\Gamma}}\rightarrow 0,\\
&\Vert u_{n}\Vert^{2}_{\lambda_{n}, \Omega'_{j}}\rightarrow \int_{\Omega_{j}}\big (|\nabla u|^2+|u|^2\big)dx,\quad\text{for all}\,\, j\in \Gamma.
\endaligned \eqno
$$

\end{proof}

With little modifications in the arguments explored in the proof of Proposition \ref{prop1} and using Corollary \ref{coro}, we also have the result below that will be used in Section 3. \\

\begin{proposition}\label{prop1*}
	Let $u_{n} \in E_{\lambda_n,R_n}$ be a $(PS)_{\infty, R_n}$ sequence with $R_n \to +\infty$, that is,
	$$\aligned
	& u_{n}\in E_{\lambda_n,R_n}\,\, \text{and}\,\, \lambda_{n}\rightarrow\infty,\\
	&\Phi_{\lambda_{n},R_n}(u_n)\rightarrow c,\\
	&\Vert \Phi'_{\lambda_{n},R_n}(u_n)\Vert \rightarrow 0.
	\endaligned \eqno$$
	Then, for some subsequence, still denoted by $(u_{n})$,
	there exists $u\in H^{1}_0(\Omega)$ such that
	$$
	u_{n}\rightharpoonup u\quad\text{in}\,\, H^{1}(\mathbb{R}^N).
	$$
	Moreover,\\
	\noindent (i)
	$$
	\|u_n-u\|_{\lambda_n,R_n} \to 0,
	$$
	and so,
	$$
	u_{n}\rightarrow u\quad \text{in}\,\, H^{1}(\mathbb{R}^N).
	$$
	\noindent (ii)  $u\equiv0$ in $\mathbb{R}^N\backslash \Omega_{\Gamma}$ and $u$ is a solution of
	$$
	\left\{
	\begin{array}{lc}
	-\Delta u=u \log u^2, & \mbox{in} \quad \Omega_{\Gamma}, \\
	u=0,  & \mbox{on} \quad \partial \Omega_{\Gamma}.  \\
	\end{array}
	\right.
	\eqno{(P_{\infty, \Gamma})}
	$$

	\noindent (iii) $u_{n}$ also satisfies
	$$\aligned
	& \lambda_{n}\int_{B_{R_n}(0)} V(x)|u_{n}|^2dx \rightarrow 0,\\
	&\Vert u_{n}\Vert^{2}_{\lambda_{n},B_{R_n}(0)\backslash \Omega_{\Gamma}}\rightarrow 0,\\
	&\Vert u_{n}\Vert^{2}_{\lambda_{n}, \Omega'_{j}}\rightarrow \int_{\Omega_{j}}\big (|\nabla u|^2+|u|^2\big)dx,\quad\text{for all}\,\, j\in \Gamma.
	\endaligned \eqno$$

\end{proposition}
\begin{proof}
First of all, the boundedness of $(\Phi_{\lambda_{n},R_n}(u_n))$ implies that there exists $K>0$ such that
$$
\Vert u_{n}\Vert^{2}_{\lambda_{n}, R_{n}}\leq K, \quad \forall n\in \mathbb{N}.
$$
Thus $(u_{n})$ is bounded in $H^{1}(\mathbb{R}^N)$ and we can assume that for some $u\in H^{1}(\mathbb{R}^N)$
$$
u_{n}\rightharpoonup u\quad\text{in}\,\, H^{1}(\mathbb{R}^N),
$$
and
$$
u_{n}(x)\rightarrow u(x)\quad\text{a.e. in}\,\, \mathbb{R}^N,
$$
and $u(x)=0$ on $\mathbb{R}^N\backslash \overline{\Omega}$.\\
$(i)$ Let $0<R<R_n$ and  $\phi_{R}\in C^{\infty}(\mathbb{R}^{N}, \mathbb{R})$ be a cut-off function
	such that
	\begin{equation*}
	\phi_{R}=0\quad x\in B_{R/2}(0), \quad \phi_{R}=1\quad x\in B^{c}_{R}(0),\quad 0\leq \phi_{R}\leq 1,\quad\text{and}\, \quad\vert \nabla\phi_{R}\vert\leq C/R,
	\end{equation*}
	where $C>0$ is a constant independent of $R$. Since the sequence $(\|\phi_{R}u_{n}\|_{\lambda_{n}, R_n})$ is bounded,  we derive that
	\begin{equation*}
	\Phi_{\lambda_n, R_n}'( u_{n})(\phi_{R} u_{n})=o_{n}(1),
	\end{equation*}
	that is
	\begin{eqnarray*}
		\int\Big(\vert \nabla u_{n} \vert^{2}+(\lambda_{n} V(x)+1)\vert u_{n} \vert^{2}\Big)\phi_{R}dx&=&\int_{\Omega'_{\Gamma}} F'_{2}(u_{n}^{+})\phi_{R}u_{n}dx+\int_{\mathbb{R}^N\backslash \Omega'_{\Gamma}} \widetilde{F}'_{2}(u_{n}^{+})\phi_{R}u_{n}dx\\
		&&-\int  u_{n} \nabla u_{n}\nabla \phi_{R}dx -\int F'_{1}(u_{n})\phi_{R}u_{n}dx+o_{n}(1).
	\end{eqnarray*}
	Choosing $R>0$ such that $\Omega'_{\Gamma} \subset B_{R/2}(0)$, the H\"{o}lder inequality together with the boundedness of the sequence $(\|u_{n}\|_{\lambda_{n}, R_n})$ in $\mathbb{R}$ leads to
	\begin{eqnarray*}
		\int\Big(\vert \nabla u_{n}\vert^{2}+(\lambda_{n} V(x)+1)\vert u_{n} \vert^{2}\Big)\phi_{R}dx&\leq&l\int \vert u_{n}\vert^{2}\phi_{R}dx+\frac{C}{R}\Vert u_{n}\Vert_{\lambda_{n}, R_n}^{2}+o_{n}(1).
	\end{eqnarray*}
	So, fixing $\zeta>0$ and passing to the limit in the last inequality, it follows that
\begin{align}\label{uniformly}
	\underset{n\rightarrow\infty}{\lim\sup}\int_{\mathbb{R}^{N}\backslash B_{R}(0)}(\vert \nabla u_{n}\vert^{2}+(\lambda_{n} V(x)+1)\vert u_{n}\vert^{2})dx\leq \frac{C}{R}< \zeta,
	\end{align}
	for some $R$ sufficiently large.\\
Since $G'_2$ has a subcritical growth, the above estimate \eqref{uniformly} ensures that
	$$
	\int G'_2(x, u_{n}^{+})w\,dx \to \int G'_2(x, u^{+})w\,dx, \quad \forall w \in C_{0}^{\infty}(\mathbb{R}^N),
	$$
	$$
	\int G'_2(x, u_{n}^{+})u_{n}^{+}\,dx \to \int G'_2(x, u^{+})u^{+}\,dx,
	$$
	and
	$$
	\int G_2(x,u_{n}^{+})\,dx \to \int G_2(x, u^{+})\,dx.
	$$
	Now, recalling that $\displaystyle \lim_{n\rightarrow\infty}\Phi_{\lambda_n, R_n}'(u_{n})w=0$ for all $w \in C_{0}^{\infty}(\mathbb{R}^N)$ and $\Vert u_{n}\Vert^{2}_{\lambda_{n}, R_{n}}\leq K$, $\forall n\in \mathbb{N}$, we deduce that
$$
\int \big (\nabla u \nabla \omega +u\,\omega\big)dx+\int F'_1(u)\omega\,dx=\int G'_2( x, u^{+})\omega\,dx,
$$
and so,
	$$
	\int \big (|\nabla u|^2+|u|^2\big)dx+\int F'_1(u)u\,dx=\int G'_2( x, u^{+})u^{+}\,dx.
	$$
This together with the equality $\displaystyle \lim_{n\rightarrow\infty}\Phi_{\lambda_{n}, R_n}'(u_{n})u_{n}=0$, that is,
	$$
	\displaystyle \int  \big (|\nabla u_{n}|^2+(\lambda_{n} V(x) +1)|u_{n}|^2\big)dx+\int F'_1(u_{n})u_{n}\,dx= \displaystyle \int  G'_2(x, u_{n}^{+})u_{n}^{+} dx+o_{n}(1),
	$$
leads to
	$$
	\begin{array}{l}
\displaystyle \lim_{n\to +\infty}	\displaystyle \Big(\int \big (|\nabla u_{n}|^2+(\lambda_{n} V(x) +1)|u_{n}|^2\big)dx+\int F'_1(u_{n})u_{n}\,dx\Big)= \\
\displaystyle \int \big (|\nabla u|^2+|u|^2\big)dx+\int F'_1(u)u\,dx,
\end{array}
	$$
	from where it follows that, for some subsequence,
	$$
	u_{n} \to u \quad \mbox{in} \quad H^{1}(\mathbb{R}^N),
	$$
$$
\lambda_n\int V(x)|u_n|^{2}\,dx \to 0,
$$
and
$$
	F'_1(u_{n})u_{n} \to F'_1(u)u \quad \mbox{in} \quad L^{1}({\mathbb{R}}^N).
$$
	Since $F_1$ is convex, even and $F(0)=0$, we know that $F'_1(t)t \geq F_1(t)\geq 0$ for all $t \in {\mathbb{R}}$. Thus, the last limit together with the Lebesgue Dominated Convergence Theorem yields
	$$
	F_1(u_{n}) \to F_1(u) \quad \mbox{in} \quad L^{1}({\mathbb{R}}^N).
	$$
Since
$$
\|u_n-u\|^{2}_{\lambda_n,R_n}=\int |\nabla u_n - \nabla u|^{2}\,dx+\int | u_n - u|^{2}\,dx+\lambda_n\int V(x)|u_n|^{2}\,dx,
$$
we also have that
$$
\|u_n-u\|^{2}_{\lambda_n,R_n} \to 0,
$$
finishing the proof of $(i)$. \\
The proofs of $(ii)$ and $(iii)$ are similar to that of Proposition \ref{prop1}, so we omit it.
\end{proof}

\subsection{The $L^{\infty}$-boundedness of  the $(M_{\lambda, R})$ solutions}\label{tpm}
In this subsection, we study the boundedness outside $\Omega'_{\Gamma}$ for some solutions of $(M_{\lambda, R})$.

\begin{lemma}\label{le511}
Let $(u_{\lambda, R})$ be a family of positive solutions of $(M_{\lambda, R})$ such that  $(\Phi_{\lambda,R}(u_{\lambda, R}))$ is  bounded in $\mathbb{R}$
for any $\lambda> 0$ and $R>0$ large. Then, there exist $K>0$ that does not depend on $\lambda> 0$ and $R>0$, and $R^*>0$ such that
$$
\vert u_{\lambda,R}\vert_{\infty, R}\leq K, \quad \forall \lambda>0 \quad \mbox{and} \quad R \geq R^*.
$$
\end{lemma}

\begin{proof}
For each  $\lambda> 0$, $L>0$ and $\beta>1$, let
\begin{equation*}
u_{L, \lambda}:=\left\{
\begin{array}{l}
u_{\lambda,R},\quad \quad \quad \text{if}\,\,u_{\lambda,R}\leq L\\
L ,\quad\quad\quad\,\, \,\,\text{if}\,\,u_{\lambda,R}\geq L.
\end{array}%
\right.
\end{equation*}
$$
z_{L, \lambda}=u_{L, \lambda}^{2(\beta-1)}u_{\lambda,R}\quad\text{and}\quad  \omega_{L, \lambda}=u_{\lambda,R}u_{L, \lambda}^{\beta-1}.
$$
Using the fact that $(u_{\lambda, R})$ is a  positive solution to $(M_{\lambda, R})$ and taking
$z_{L, \lambda}$ as a test function, we have
\begin{eqnarray}\label{ine31}
&&\int_{B_R(0)} u_{L, \lambda}^{2(\beta-1)}\vert \nabla u_{\lambda,R}\vert^{2}dx+2(\beta-1)\int_{B_R(0)} u_{L, \lambda}^{2\beta-3}u_{\lambda,R} \nabla u_{\lambda,R}\nabla u_{L, \lambda} dx\\
&&+\int_{B_R(0)}(\lambda V(x)+1)u_{L, \lambda}^{2(\beta-1)}\vert  u_{\lambda,R}\vert^{2}dx\nonumber+\int_{B_R(0)} F'_{1}(u_{\lambda,R})u_{L, \lambda}^{2(\beta-1)}u_{\lambda,R}dx=\int_{B_R(0)} G'_{2}(x, u_{\lambda,R})u_{L, \lambda}^{2(\beta-1)} u_{\lambda,R}dx.
\end{eqnarray}
From the definition of $G_{2}$,
\begin{eqnarray}\label{ine32}
G'_{2}( x, t)\leq F'_{2}(t)\leq Ct^{p-1}, \quad  \forall \, (x, t)\in \mathbb{R}^N\times \mathbb{R}^{+},
\end{eqnarray}
where $p\in (2, 2^{*})$.
Hence, from \eqref{ine31} and \eqref{ine32},
\begin{eqnarray}\label{ine33}
\int_{B_R(0)}( \vert \nabla \omega_{L, \lambda}\vert^{2}+\vert  \omega_{L, \lambda}\vert^{2} )dx\leq C\int_{B_R(0)} u_{\lambda,R}^{p}u_{L, \lambda}^{2(\beta-1)} dx= C\int_{B_R(0)} u_{\lambda,R}^{p-2}\omega_{L, \lambda}^{2} dx.
\end{eqnarray}
Using the H\"{o}lder inequality,
\begin{eqnarray}\label{ine73}
\int_{B_R(0)} u_{\lambda, R}^{p-2}\omega_{L, \lambda}^{2} dx&\leq& C\beta^{2}\Big(\int_{B_R(0)} u_{\lambda, R}^{p}dx\Big)^{(p-2)/p}\Big(\int_{B_R(0)} \omega_{L, \lambda} ^{p}dx\Big)^{2/p}.
\end{eqnarray}
On the other hand, by the Sobolev inequality,
\begin{eqnarray}\label{ine35}
\Big(\int_{B_R(0)}\vert \omega_{L, \lambda} \vert^{2^{*}} dx\Big)^{2/2^{*}} \leq C\int_{B_R(0)}( \vert \nabla \omega_{L, \lambda}\vert^{2}+\vert  \omega_{L, \lambda}\vert^{2} )dx.
\end{eqnarray}
Combining \eqref{ine33}, \eqref{ine73} and \eqref{ine35},
\begin{eqnarray*}\label{ine36}
\Big(\int_{B_R(0)}\vert \omega_{L, \lambda} \vert^{2^{*}} dx\Big)^{2/2^{*}}
\leq C\beta^{2}\Big(\int_{B_R(0)} u_{\lambda} ^{p\beta}dx\Big)^{2/p}.
\end{eqnarray*}
Using the Fatou's lemma in the variable $L$, one has
\begin{eqnarray*}\label{ine36}
\Big(\int_{B_R(0)}\vert u_{\lambda} \vert^{2^{*}\beta} dx\Big)^{2/2^{*}}
\leq C\beta^{2}\Big(\int_{B_R(0)} u_{\lambda} ^{p\beta}dx\Big)^{2/p},
\end{eqnarray*}
and so,
\begin{eqnarray}\label{ine366}
\Big(\int_{B_R(0)}\vert u_{\lambda} \vert^{2^{*}\beta} dx\Big)^{1/2^{*}\beta}
\leq C^{1/\beta}\beta^{^{1/\beta}}\Big(\int_{B_R(0)} u_{\lambda} ^{p\beta}dx\Big)^{1/p\beta}.
\end{eqnarray}
Since $(\Phi_{\lambda,R}(u_{\lambda, R}))$ is  bounded in $\mathbb{R}$
for any $\lambda> 0$ and $R>0$ large and $(u_{\lambda, R})$ is a solution of $(M_{\lambda, R})$, arguing as in Subsection 2.1 there exists $C>0$ such that
$$
\|u_{\lambda,R}\|_{\lambda,R} \leq C
$$
$\lambda> 0$ and $R>0$ large. Fixing any sequences $\lambda_n \to +\infty$ and $R_n \to +\infty$, we may see that $(u_{\lambda_n, R_n})$ satisfies the hypotheses from Proposition \ref{prop1*}, then $u_{\lambda_n,R_n}\rightarrow u$ strongly in $H^{1}(\mathbb{R}^N)$. Now, since $2<p<2^{*}$ and $(|u_{\lambda_n,R_n}|_{L^{2^*}(\mathbb{R}^N)})$ is bounded in $\mathbb{R}$, a well known iteration argument (see in \cite[Lemma 3.10]{AlvesChao2}) and \eqref{ine366} implies that there exists a positive constant $K_1>0$ such that
$$
\vert u_{\lambda_n,R_n}\vert_{L^{\infty}(\mathbb{R}^N)}\leq K_1,\quad \forall n \in \mathbb{N}.
$$
From the above analysis, it is easy to see that the lemma follows arguing by contradiction.
\end{proof}
\begin{lemma}\label{le51}
Let $(u_{\lambda, R})$ be a family of positive solutions of $(M_{\lambda, R})$ with $\Phi_{\lambda,R}(u_{\lambda, R})$  bounded in $\mathbb{R}$
for any $\lambda> 0$ and $R>0$ large. Then, there exist $\lambda^{\prime}>0$ and $R^{\prime}>0$ such that
$$
\vert u_{\lambda, R}\vert_{\infty, B_R(0)\backslash \Omega'_{\Gamma}}\leq a_{0}, \quad \forall \lambda\geq \lambda^{\prime},\,\, R\geq R^{\prime}.
$$
In particular, $u_{\lambda, R}$ solves the original problem $(P_{\lambda, R})$ for any $\lambda\geq \lambda^{\prime}$ and  $R\geq R^{\prime}$.
\end{lemma}

\begin{proof}
Choose $R_0>0$ large such that $ \overline{\Omega'_{\Gamma}} \subset B_{R_0}(0)$ and fix  a neighborhood $\mathcal{B}$ of $\partial \Omega'_{\Gamma}$ such that
$$
\mathcal{B}\subset B_{R_0}(0)\backslash \Omega_{\Gamma}.
$$
The Moser's iteration technique implies that there exists $C>0$, which is independent of $\lambda$, such that
$$
\vert u_{\lambda, R}\vert_{L^{\infty}(\partial \Omega'_{\Gamma})}\leq C\vert u_{\lambda, R}\vert_{L^{2^{*}}(\mathcal{B})}, \quad \forall\, R \geq R_0.
$$
Fixing two sequences $\lambda_n \to +\infty$, $R_n \to +\infty$ and using Proposition \ref{prop1*},  we have for some subsequence  that $u_{\lambda_n, R_n}\rightarrow 0$ in $H^{1}(B_{R_n}(0)\backslash \Omega_{\Gamma})$, then $u_{\lambda_n, R_n}\rightarrow 0$ in $H^{1}(B_{R_0}(0)\backslash \Omega_{\Gamma})$, and so,
$$
\vert u_{\lambda_n, R_n}\vert_{L^{2^{*}}(\mathcal{B})}\rightarrow 0\,\,\,\text{as}\,\, n\rightarrow\infty.
$$

Hence, there exists $n_0 \in\mathbb{N}$ such that
$$
\vert u_{\lambda_n, R_n}\vert_{L^{\infty}(\partial \Omega'_{\Gamma})}\leq a_{0},\,\,\, \forall\, n \geq n_0.
$$
Now, for $n \geq n_0$ we set $\widetilde{u}_{\lambda_n, R_n}: B_{R_n}(0)\backslash \Omega'_{\Gamma}\rightarrow \mathbb{R}$ given by
\begin{align*}
\tilde{u}_{\lambda_n, R_n}(x)=(u_{\lambda_n, R_n}-a_{0})^{+}(x).
\end{align*}
Thereby, $\tilde{u}_{\lambda_n, R_n}(x)\in H_{0}^{1}(B_{R_n}(0)\backslash \Omega'_{\Gamma})$. Our goal is to show that $\tilde{u}_{\lambda_n, R_n}(x)=0$
in $B_{R_n}(0)\backslash \Omega'_{\Gamma}$, because this will ensure that
$$
\vert u_{\lambda_n, R_n}\vert_{\infty,\, B_{R_n}(0)\backslash \Omega'_{\Gamma}}\leq a_{0}.
$$
In fact, extending $\tilde{u}_{\lambda_n, R_n}(x)=0$
in $\Omega'_{\Gamma}$ and taking $\tilde{u}_{\lambda, R}$ as a test function, we obtain
\begin{equation*}
\int_{B_{R_n}(0)\backslash\Omega'_{\Gamma}}\nabla u_{\lambda_n, R_n}\nabla \tilde{u}_{\lambda_n, R_n}dx+\int_{B_{R_n}(0)\backslash\Omega'_{\Gamma}}(\lambda_n V(x)+1)u_{\lambda_n, R_n}\tilde{u}_{\lambda_n, R_n}dx\leq \int_{B_{R_n}(0)\backslash\Omega'_{\Gamma}}\tilde{F}'_2( u_{\lambda_n, R_n}))\tilde{u}_{\lambda_n, R_n}dx.
\end{equation*}
Since
\begin{eqnarray*}
&& \int_{B_{R_n}(0)\backslash\Omega'_{\Gamma}}\nabla u_{\lambda_n, R_n}\nabla \tilde{u}_{\lambda_n, R_n}dx=\int_{B_{R_n}(0)\backslash\Omega'_{\Gamma}}\vert \nabla \tilde{u}_{\lambda_n, R_n}\vert^{2} dx,\\
&&\int_{B_{R_n}(0)\backslash\Omega'_{\Gamma}}(\lambda_n V(x)+1)u_{\lambda_n, R_n}\tilde{u}_{\lambda_n, R_n}dx=\int_{(B_{R_n}(0)\backslash\Omega'_{\Gamma})_{+}}(\lambda_n V(x)+1)(\tilde{u}_{\lambda_n, R_n}+a_{0})\tilde{u}_{\lambda_n, R_n}dx,
\end{eqnarray*}
and
\begin{equation*}
\int_{B_{R_n}(0)\backslash\Omega'_{\Gamma}}\tilde{F}'_2( u_{\lambda_n, R_n})\tilde{u}_{\lambda_n, R_n}dx=\int_{(B_{R_n}(0)\backslash\Omega'_{\Gamma})_{+}}\frac{\tilde{F}'_2( u_{\lambda_n, R_n})}{u_{\lambda_n, R_n}}(\tilde{u}_{\lambda_n, R_n}+a_{0})\tilde{u}_{\lambda_n, R_n},
\end{equation*}
where
\begin{equation*}
(B_{R_n}(0)\backslash\Omega'_{\Gamma})_{+}=\{x\in B_{R_n}(0)\backslash\Omega'_{\Gamma}: u_{\lambda_n, R_n}(x)>a_{0}\}.
\end{equation*}
From the above equalities,
\begin{equation*}
\int_{B_{R_n}(0)\backslash\Omega'_{\Gamma}}\vert \nabla \tilde{u}_{\lambda_n, R_n}\vert^{2} dx+\int_{(B_{R_n}(0)\backslash\Omega'_{\Gamma})_{+}}\Big((\lambda_n V(x)+1)-\frac{\tilde{F}'_2( u_{\lambda_n, R_n})}{u_{\lambda_n, R_n}}\Big)(\tilde{u}_{\lambda_n, R_n}+a_{0})\tilde{u}_{\lambda_n, R_n}=0.
\end{equation*}
By the definition of $\tilde{F}'_2$, we know that
\begin{equation*}
(\lambda_n V(x)+1)-\frac{\tilde{F}'_2( u_{\lambda_n, R_n})}{u_{\lambda_n, R_n}}\geq 1-l> 0\quad\text{in}\,\,(B_{R_n}(0)\backslash\Omega'_{\Gamma})_{+}.
\end{equation*}
Thus, $\tilde{u}_{\lambda_n, R_n}=0$ in $(B_{R_n}(0)\backslash\Omega'_{\Gamma})_{+}$, and $\tilde{u}_{\lambda_n, R_n}=0$ in $B_{R_n}(0)\backslash\Omega'_{\Gamma}$.
From the above arguments, there exist $\lambda^{\prime}>0$ and $R^{\prime}>0$  such that
$$
\vert u_{\lambda, R}\vert_{\infty, B_R(0)\backslash \Omega'_{\Gamma}}\leq a_{0}, \quad \forall \lambda\geq \lambda^{\prime},\,\, R\geq R^{\prime}.
$$
The proof is complete.
\end{proof}

\subsection{A special minimax level}\label{tpm}

In this subsection, for any $\lambda>0$ and $j\in \Gamma$, let us denote by $I_{j}:  H_{0}^{1}(\Omega_{j})\rightarrow \mathbb{R}$ and $I_{\lambda, j}:H^{1}(\Omega'_{j})\rightarrow \mathbb{R}$  the functionals given by
\begin{equation*}
I_{j}(u)=\frac{1}{2}\int_{\Omega_{j}}\big (|\nabla u|^2+|u|^2\big)dx-\frac{1}{2}\int_{\Omega_{j}}u^2\log u^{2} \,dx,
\end{equation*}

\begin{equation*}
I_{\lambda, j}(u)=\frac{1}{2}\int_{\Omega'_{j}}\big (|\nabla u|^2+(\lambda V(x)+1)|u|^2\big)dx-\frac{1}{2}\int_{\Omega'_{j}}u^2\log u^{2} \,dx,
\end{equation*}
which are the energy functionals associated to  the following logarithmic equations
$$
\left\{
\begin{array}{lc}
-\Delta u=u \log u^2, \quad\, \text{in} \,\, \Omega_{j}, \\
u=0,  \quad\quad\quad\quad\quad\,\,\, \text{on} \,\, \partial \Omega_{j},  \\
\end{array}
\right.
\eqno{(D_{j})}
$$
and
$$
\left\{
\begin{array}{lc}
-\Delta u+ \lambda V(x)u=u \log u^2, \quad\,\text{in}\,\in\Omega'_{j},\\
\frac{\partial u}{\partial \eta}=0, \quad \quad\quad\quad\quad\quad\quad\quad\quad\,\,\,\text{on}\,\, \partial \Omega'_{j}.
\end{array}
\right.
\eqno{(N_{j})}
$$
It is immediate to check that $I_{j}$ and $I_{\lambda, j}$ satisfy the mountain pass geometry. Since $\Omega_{j}$ and $\Omega'_{j}$ are bounded, $I_{j}$ and $I_{\lambda, j}$
satisfy the Palais-Smale condition, from Mountain Pass Theorem due to Ambrosetti-Rabinowitz \cite{AmbRabi}, there exist two positive functions $\omega_{j}\in H^{1}_{0}(\Omega_{j})$ and $\omega_{\lambda, j}\in H^{1}(\Omega'_{j})$ verifying
$$
I_{j}(\omega_{j})=c_{j},\quad I_{\lambda, j}(\omega_{\lambda, j})=c_{\lambda, j}\,\,\, \text{and}\,\,\,I'_{j}(\omega_{j})=I'_{\lambda, j}(\omega_{\lambda, j})=0,
$$
where
$$
c_{j}=\inf_{\gamma\in \Upsilon_{j}}\max_{t\in [0, 1]}I_{j}(\gamma(t)),
$$
$$
c_{\lambda, j}=\inf_{\gamma\in \Upsilon_{\lambda, j}}\max_{t\in [0, 1]}I_{\lambda, j}(\gamma(t)),
$$
$$
\Upsilon_{j}=\{\gamma\in C([0, 1], H^{1}_{0}(\Omega_{j})): \gamma(0)=0, \,\, \text{and}\,\, I_{j}(\gamma(1))<0\},
$$
and
$$
\Upsilon_{\lambda, j}=\{\gamma\in C([0, 1], H^{1}(\Omega'_{j})): \gamma(0)=0, \,\, \text{and}\,\, I_{\lambda, j}(\gamma(1))<0\}.
$$
In fact, a simple computation gives
$$
c_{j}=\inf_{u\in \mathcal{N}_{j}}I_{j}(u),
$$
$$
c_{\lambda, j}=\inf_{u\in \mathcal{N}'_{j}}I_{\lambda, j}(u),
$$
where
$$
\mathcal{N}_{j}=\left\{u \in H_0^{1}(\Omega_{j}) \backslash \{0\}: I'_{j}(u)u=0  \right\},
$$
and
$$
\mathcal{N}'_{j}=\left\{u \in H^{1}(\Omega'_{j}) \backslash \{0\}: I'_{\lambda, j}(u)u=0  \right\}.
$$

Moreover, by a direct computation, there exists $\tau>0$ such that if $u\in \mathcal{N}_{j}$, for any $j\in \Gamma$, then
\begin{equation}\label{below}
\Vert u\Vert_{j}>\tau,
\end{equation}
where $\Vert \,\, \Vert_{j}$ denotes the norm on $H_{0}^{1}(\Omega_{j})$ given by
$$
\Vert u \Vert_{j}=\Big(\int_{\Omega_{j}}(\vert \nabla u\vert^{2}+ \vert  u\vert^{2})dx\Big)^{\frac{1}{2}}.
$$
In particular, as  $\omega_{j}\in \mathcal{N}_{j}$, we must have
\begin{equation*}\label{below1}
\Vert\omega_{\lambda, j}\Vert_{j}>\tau,
\end{equation*}
where $\omega_{\lambda, j}=\omega_{j}\mid_{\Omega_{j}}$ for all $j\in \Gamma$.

In what follows, $c_{\Gamma}=\sum_{j=1}^{l}c_{j}$ and $T>0$ is a constant large enough, which does not depend on $\lambda$ and $R>0$ large enough, such that
\begin{equation}\label{4.1}
0<I'_{j}(\frac{1}{T}\omega_{j})(\frac{1}{T}\omega_{j}),\quad I'_{j}(T \omega_{j})(T \omega_{j})<0\,\, \,\,\forall j\in \Gamma.
\end{equation}
Hence, by the definition of $c_{j}$, one has
$$
\max_{s\in [1/T^{2}, 1]}I_{j}(sT\omega_{j})=c_{j}, \,\, \,\,\forall j\in \Gamma.
$$
Without loss of generality, we consider $\Gamma=\{1, 2, \cdots, l\}$, with $l\leq k$ and fix
$$
\gamma_{0}(s_{1}, s_{2}, \cdots, s_{l})(x)=\sum_{j=1}^{l}s_{j}T\omega_{j}(x),\quad \forall (s_{1}, s_{2}, \cdots, s_{l})\in [1/T^{2}, 1]^{l},
$$
$$
\Gamma_{*}=\{\gamma\in C([1/T^{2}, 1]^{l}, E_{\lambda,R}\backslash \{0\}): \gamma=\gamma_{0}\,\,\,\text{on}\,\, \partial ([1/T^{2}, 1]^{l})\},
$$
and
$$
b_{\lambda,R, \Gamma}=\inf_{\gamma\in \Gamma_{*}}\max_{(s_{1}, s_{2}, \cdots, s_{l})\in [1/T^{2}, 1]^{l}}\Phi_{\lambda, R}(\gamma(s_{1}, s_{2}, \cdots, s_{l})).
$$
We remark that $\gamma_{0}\in \Gamma_{*}$, so  $\Gamma_{*}\neq \emptyset$ and $b_{\lambda,R, \Gamma}$ is well-defined.

\begin{lemma}\label{Lem:3.2a}
For each $\gamma\in \Gamma_{*}$, there exists $(t_{1}, t_{2}, \cdots, t_{l})\in [1/T^{2}, 1]^{l}$ such that
$$
I'_{\lambda, j}(\gamma(t_{1}, \cdots, t_{l}))\gamma(t_{1}, \cdots, t_{l})=0, \,\, \text{for}\,\, j\in\{1, \cdots, l\}.
$$
\end{lemma}
\begin{proof}
Given $\gamma\in \Gamma_{*}$, consider the map $\widetilde{\gamma}: [1/T^{2}, 1]^{l}\rightarrow \mathbb{R}^l$ defined by
$$
\widetilde{\gamma}(s_{1},  \cdots, s_{l})=\Big(I'_{\lambda, 1}(\gamma(s_{1}, \cdots, s_{l}))\gamma(s_{1}, \cdots, s_{l}), \cdots, I'_{\lambda, l}(\gamma(s_{1}, \cdots, s_{l}))\gamma(s_{1}, \cdots, s_{l})\Big).
$$
For $(s_{1},  \cdots, s_{l})\in \partial ([1/T^{2}, 1]^{l})$, we know that
$$
\gamma(s_{1},  \cdots, s_{l})=\gamma_{0}(s_{1},  \cdots, s_{l}).
$$
Now, the lemma follows by employing \eqref{4.1} and Miranda's Theorem \cite{M}.
\end{proof}

\begin{lemma} \label{lemma1}
	\noindent (a)  $\sum_{j=1}^{l}c_{\lambda, j}\leq b_{\lambda,R, \Gamma}\leq c_{\Gamma}$, for any $\lambda> 0$ and $R>0$ large enough;

	\noindent (b) For $\gamma\in \Gamma_{*}$ and $(s_{1},  \cdots, s_{l})\in \partial ([1/T^{2}, 1]^{l})$, we have
$$
\Phi_{\lambda, R}(\gamma(s_{1},  \cdots, s_{l}))< c_{\Gamma},\,\, \forall \lambda> 0.
$$
\end{lemma}
The proof of the lemma is the same as that of Proposition 4.2 in \cite{Alves}, so we omit it.

\vspace{0.5 cm}

\begin{corollary}\label{coro1}
\noindent (a)  $b_{\lambda,R,\Gamma}$ is a critical value of $\Phi_{\lambda, R}$ for $\lambda>0$ and $R>0$ large enough.

\noindent (b)  $b_{\lambda,R, \Gamma}\rightarrow c_{\Gamma}$, when $\lambda\rightarrow +\infty$ uniform for $R>0$ large.

\end{corollary}

The proof of the corollary is similar to that of Corollary 4.3 in \cite{Alves}, here we also omit it.

\subsection{A special solution for the auxiliary problem}

Hereafter, let us denote by $\Phi_{\lambda, R}^{c_{\Gamma}}$ and $\Upsilon$ the sets below
\begin{align*}
	\Phi_{\lambda, R}^{c_{\Gamma}}=\{u\in E_{\lambda,R}:  \Phi_{\lambda, R}(u)\leq c_{\Gamma}\}
\end{align*}
and
$$
\Upsilon=\Big\{ u\in E_{\lambda,R}: \Vert u\Vert_{\lambda, \Omega'_{j}}>\frac{\tau}{2T}, \,\, \forall j\in\Gamma \Big\}
$$
where $\tau,T$ were fixed in \eqref{below} and \eqref{4.1} respectively.\\

Fixing $\kappa=\frac{\tau}{8 T}$ and $\mu>0$, we define
\begin{align*}
A_{\mu,R}^{\lambda}=\{u\in \Upsilon_{2\kappa} : \Phi_{\lambda, B_R(0)\backslash \Omega'_{\Gamma}}(u)\geq 0, \,\Vert u\Vert^{2}_{\lambda,B_R(0)\backslash \Omega_{\Gamma}}\leq\mu, \,\,\vert I_{\lambda, j}(u)-c_{j}\vert\leq \mu, \,\, \forall j\in \Gamma\}
\end{align*}
where $\Upsilon_{r}$, for $r>0$, denotes the set $\Upsilon_{r}=\{u\in E_{\lambda,R}: \inf_{v\in \Upsilon}\Vert u-v\Vert_{\lambda, \Omega'_{j}}\leq r,\,\forall j\in\Gamma$\}.  Notice that $w=\sum_{j=1}^{l}w_{j}\in A_{\mu,R}^{\lambda}\cap  \Phi_{\lambda, R}^{c_{\Gamma}}$ which shows that $A_{\mu,R}^{\lambda}\cap  \Phi_{\lambda, R}^{c_{\Gamma}}\neq\emptyset$.

Next, we shall establish a very important uniform estimate of $\Vert \Phi'_{\lambda, R}(u)\Vert$ in the set $(A_{2\mu,R}^{\lambda}\backslash A_{\mu,R}^{\lambda})\cap  \Phi_{\lambda, R}^{c_{\Gamma}}$.

\begin{proposition}\label{prop71}
For each $\mu>0$, there exist $\Lambda_{*}> 0$ and $R^*>0$ large enough  and $\sigma_{0}>0$ independent of $\lambda$ and $R>0$ large  such that
\begin{align*}
\Vert \Phi'_{\lambda, R}(u)\Vert\geq \sigma_{0}\quad \text{for}\,\lambda\geq \Lambda_{*}, R \geq R^* \,\,\text{and}\,\, u\in (A_{2\mu,R}^{\lambda}\backslash A_{\mu,R}^{\lambda})\cap  \Phi_{\lambda, R}^{c_{\Gamma}}.
\end{align*}
\end{proposition}
\begin{proof}
Arguing by contradiction, we assume that there exist $\lambda_{n},R_n\rightarrow \infty$ and $u_{n}\in (A_{2\mu,R_n}^{\lambda_{n}}\backslash A_{\mu,R_n}^{\lambda_{n}})\cap  \Phi_{\lambda_{n}, R_n}^{c_{\Gamma}}$
such that
$$\Vert \Phi'_{\lambda_{n}, R_n}(u_{n})\Vert\rightarrow 0.$$
Since $u_{n}\in A_{2\mu,R_n}^{\lambda_{n}}$, we know that $(\Vert u_{n}\Vert_{\lambda_{n}, R_n})$ and
$(\Phi_{\lambda_{n}, R_n}(u_{n}))$ are both bounded. Then, passing to a subsequence if necessary, we can assume that $(\Phi_{\lambda_{n}, R_n}(u_{n}))$ is a convergent sequence.
Thus, from Proposition \ref{prop1*}, there exists $0\leq u\in H_{0}^{1}(\Omega_{\Gamma})$ such that $u$ is a solution for $(D_{j})$ and
$$
u_{n}\rightarrow u\quad \text{in}\,\, H^{1}(\mathbb{R}^N),\quad \Vert u_{n}\Vert^{2}_{\lambda_{n},B_{R_n}(0)\backslash \Omega_{\Gamma}}\rightarrow 0,\,\, \text{and}\,\, \Phi_{\lambda_{n}, R_n}(u_{n})\rightarrow I_{\Gamma}(u)\in (-\infty, c_{\Gamma}].
$$
As $(u_{n})\subset \Upsilon_{2\kappa}$, it occurs that
$$
\Vert u_{n}\Vert^{2}_{\lambda_{n},\, \Omega'_{j}}>\frac{\tau}{4 T}, \quad \forall j\in \Gamma,
$$
and so, letting $n\rightarrow +\infty$, we get the inequality below
$$
\Vert u\Vert^{2}_{j}\geq \frac{\tau}{4 T}>0, \,\, \quad \forall j\in \Gamma,
$$
which yields $u\mid_{\Omega_{j}}\not= 0$, $j=1, \cdots, l$ and $I'_{\Gamma}(u)=0$. Consequently, by \eqref{below}
$$
\Vert u\Vert^{2}_{j}>\frac{\tau}{2 T}>0, \,\, \quad \forall j\in \Gamma.
$$
In this way, $I_{\Gamma}(u)\geq c_{\Gamma}$.
However, from the fact that $\Phi_{\lambda_{n}, R_n}(u_{n})\leq c_{\Gamma}$ and $\Phi_{\lambda_{n}, R_n}(u_{n})\rightarrow I_{\Gamma}(u)$, as $n \to +\infty$, we deduce that
$I_{\Gamma}(u)=c_{\Gamma}$. Thus, for $n$ large enough
\begin{align*}
\Vert u_{n}\Vert^{2}_{j}> \frac{\tau}{2 T},\quad\vert \Phi_{\lambda_{n}, R_n}(u_{n})-c_{\Gamma}\vert\leq \mu,\,\, \text{for any}\,\,  j\in\Gamma.
\end{align*}
So,  $u_{n}\in  A_{\mu,R_n}^{\lambda_{n}}$ for large $n$, which is a contradiction to $u_{n}\in  (A_{2\mu,R_n}^{\lambda_{n}}\backslash A_{\mu,R_n}^{\lambda_{n}})$.
Thus, we complete the proof.
\end{proof}
In the sequel, $\mu_{1}$, $\mu^{*}$ denote the following numbers
$$
\underset{\textbf{t}\in \partial[1/T^{2}, 1]^{l}}{\min} \vert I_{\Gamma}(\gamma_{0}(\textbf{t}))-c_{\Gamma}\vert=\mu_{1}>0
$$
and
$$
\mu^{*}=\min\{\mu_{1}, \kappa, r/2\}
$$
where $\kappa=\frac{\tau}{8T}$ was given before and $r> \max\{\|w_j\|_{H^{1}_0(\Omega_j)}\,:\,j=1,...,l\}$. Moreover, for each $s>0$, $B_{s}^{\lambda}$ denotes the set
$$
B_{s}^{\lambda}=\{u\in E_{\lambda}(B_{R}(0)):  \Vert u\Vert_{\lambda, R}\leq s\}, \,\, \text{for}\,s>0.
$$

\begin{proposition}\label{prop72}
Let $\mu\in (0, \mu^{*})$ and $\Lambda_{*}> 0$ and $R^*>0$ large enough given in  Proposition \ref{prop71}.  Then, for $\lambda\geq \Lambda_{*}$ and $R \geq R^*$, there exists a
positive solution $u_{\lambda,R}$ of $(M_{\lambda, R})$ satisfying $u_{\lambda}\in A_{\mu, R}^{\lambda}\cap  \Phi_{\lambda, R}^{c_{\Gamma}}\cap B_{r+1}^{\lambda}$.
\end{proposition}
\begin{proof}
Seeking for a contradiction, let us assume that there exist no critical points for the functional $\Phi_{\lambda, R}(u)$ in $A_{\mu, R}^{\lambda}\cap  \Phi_{\lambda, R}^{c_{\Gamma}}\cap B_{r+1}^{\lambda}$ for $\lambda\geq \Lambda_{*}$. Since $\Phi_{\lambda, R}$ verifies  the (PS) condition, there exists a constant $d_{\lambda}>0$ such that
\begin{align*}
\Vert \Phi'_{\lambda, R}(u)\Vert\geq d_{\lambda}\quad \text{for all}\quad u\in A_{\mu, R}^{\lambda}\cap  \Phi_{\lambda, R}^{c_{\Gamma}}\cap B_{r+1}^{\lambda}.
\end{align*}
By Proposition \ref{prop71},
\begin{align*}
\Vert \Phi'_{\lambda, R}(u)\Vert\geq \sigma_{0}\quad\text{for all}\,\, u\in (A_{2\mu, R}^{\lambda}\backslash A_{\mu, R}^{\lambda})\cap  \Phi_{\lambda, R}^{c_{\Gamma}},
\end{align*}
where $\sigma_{0}>0$ is independent of $\lambda$. In what follows, $\Psi:E_{\lambda, R}\rightarrow \mathbb{R}$ is a continuous functional  verifying
\begin{align*}
&\Psi(u)=1\quad\quad\quad\text{for}\,\, u\in A_{3\mu/2, R}^{\lambda}\cap \Upsilon_{\kappa}\cap B_{r}^{\lambda},\\
&\Psi(u)=0\quad\quad\quad\text{for}\,\,u\not\in A_{2\mu, R}^{\lambda}\cap \Upsilon_{2\kappa}\cap B_{r+1}^{\lambda},\\
&0\leq \Psi(u)\leq 1\quad\text{for}\,\,\forall \,u\in E_{\lambda, R},
\end{align*}
and $H:\Phi_{\lambda, R}^{c_{\Gamma}}\rightarrow E_{\lambda}(B_{R}(0))$  is a function given by
\begin{equation*}
H(u):=\left\{
\begin{array}{l}
-\Psi(u)\frac{ Y(u)}{\Vert  Y(u)\Vert}, \quad \,\,u\in A_{2\mu, R}^{\lambda}\cap B_{r+1}^{\lambda} ,\\
0,  \quad\quad \quad\quad\quad\quad\,\,u\not\in A_{2\mu, R}^{\lambda}\cap B_{r+1}^{\lambda}, \\
\end{array}%
\right.
\end{equation*}
where $Y$ is a pseudo-gradient vector field for $\Phi_{\lambda, R}$ on $\mathcal{K}=\{u\in E_{\lambda,R}: \Phi'_{\lambda, R}(u)\neq 0\}$.
Observe that $H$ is well defined, since $\Phi'_{\lambda, R}(u)\neq 0$, for $u\in A_{2\mu, R}^{\lambda}\cap  \Phi_{\lambda, R}^{c_{\Gamma}}$.
The following inequality
\begin{align*}
\Vert H(u)\Vert\leq 1, \,\,\, \forall \lambda\geq \Lambda_{*}\,\,\text{and}\,\, u\in\Phi_{\lambda, R}^{c_{\Gamma}},
\end{align*}
guarantees that the deformation flow $\eta: [0, \infty)\times \Phi_{\lambda, R}^{c_{\Gamma}}\rightarrow \Phi_{\lambda, R}^{c_{\Gamma}}$ defined by
\begin{align*}
\frac{d\eta}{dt}=H(\eta)\quad \text{and}\quad \eta(0, u)=u\in \Phi_{\lambda, R}^{c_{\Gamma}},
\end{align*}
verifies
\begin{align}
&\frac{d}{dt}\Phi_{\lambda, R}(\eta(t, u))\leq-\Psi(\eta(t, u))\Vert  \Phi'_{\lambda, R}(\eta(t, u))\Vert\leq 0,\label{small1}\\
&\Vert\frac{d\eta}{dt}\Vert_{\lambda}=\Vert H(\eta)\Vert_{\lambda}\leq 1,\nonumber\\
&\eta(t, u)=u\quad\text{for all}\,\,t\geq 0\,\,\text{and}\,\, u\in \Phi_{\lambda, R}^{c_{\Gamma}}\backslash (A_{2\mu, R}^{\lambda}\cap B_{r+1}^{\lambda}).\label{small3}
\end{align}
We now study two paths, which are relevant for what follows:\\

\noindent (1) The path $\textbf{t}\rightarrow \eta(t, \gamma_{0}(\textbf{t}))$, where $\textbf{t}=(t_{1}, \cdots, t_{l})\in [1/T^{2}, 1]^{l}$.\\
Thereby, if $\mu\in (0, \mu^{*})$, we have that
$$
\gamma_{0}(\textbf{t})\not\in A_{2\mu, R}^{\lambda}, \,\, \forall \, \textbf{t}\in \partial ([1/T^{2}, 1]^{l}).
$$
Since
\begin{align*}
\Phi_{\lambda, R}(\gamma_{0}(\textbf{t}))\leq c_{\Gamma}\quad \forall\, \textbf{t}\in \partial ([1/T^{2}, 1]^{l}).
\end{align*}
From \eqref{small3}, it follows that
\begin{align*}
\eta(t, \gamma_{0}(\textbf{t}))=\gamma_{0}(\textbf{t})\quad \forall\, \textbf{t}\in \partial ([1/T^{2}, 1]^{l}).
\end{align*}
So, $\eta(t, \gamma_{0}(\textbf{t}))\in \Gamma_{*}$ for all $t\geq 0$.\\

\noindent (2) The path $\textbf{t}\rightarrow \gamma_{0}(\textbf{t})$, where $\textbf{t}=(t_{1}, \cdots, t_{l})\in [1/T^{2}, 1]^{l}$.

Since $\text{supp} (\gamma_{0}(\textbf{t}))\subset\overline{\Omega_{\Gamma}}$ for all  $\textbf{t}\in[1/T^{2}, 1]^{l}$, then $\Phi_{\lambda, R}(\gamma_{0}(\textbf{t}))$  does not depend on $\lambda> 0$. On the other hand,
\begin{align*}
\Phi_{\lambda, R}(\gamma_{0}(\textbf{t})))\leq c_{\Gamma}\quad \forall\,\,  \textbf{t}\in  [1/T^{2}, 1]^{l},
\end{align*}
and
$$
\Phi_{\lambda, R}(\gamma_{0}(\textbf{t}))= c_{\Gamma}\quad\text{if and only if}\quad t_{j}=1/T, \,\, \forall j\in \Gamma.
$$
Therefore,
\begin{align*}
m_{0}:=\sup\{\Phi_{\lambda, R}(u): u\in \gamma_{0}([1/T^{2}, 1]^{l})\backslash A_{\mu}^{\lambda}\}
\end{align*}
is independent of $\lambda,R>0$ and $m_{0}< c_{\Gamma}$. Now, observe that there exists $K_{*}>0$ such that
\begin{align*}
\vert \Phi_{\lambda, R}(u)-\Phi_{\lambda, R}(v)\vert\leq K_{*}\Vert u-v\Vert_{\lambda, R}, \quad \forall\,\,  u, v\in B_{r}^{\lambda},
\end{align*}
we claim that if  $T_*>0$ is large enough, the estimate below holds
\begin{align}\label{small5}
\max_{\textbf{t}\in[1/T^{2}, 1]^{l}}\Phi_{\lambda}(\eta(T_*, \gamma_{0}(\textbf{t})))<\max\{m_{0}, c_{\Gamma}-\frac{1}{2K_{*}}\sigma_{0}\mu\}.
\end{align}
In fact, write $u=\gamma_{0}(\textbf{t})$, $\textbf{t}\in[1/T^{2}, 1]^{l}$. If $u\not\in A_{\mu, R}^{\lambda}$, we must have by \eqref{small1},
\begin{align*}
\Phi_{\lambda, R}(\eta(t, u))\leq \Phi_{\lambda}(\eta(0, u))= \Phi_{\lambda, R}(u)\leq m_{0}, \quad \forall t\geq 0.
\end{align*}
On the other hand, if $u\in A_{\mu, R}^{\lambda}$, by setting $\tilde{\eta}(t)=\eta(t, u)$, $\tilde{d}_{\lambda}:=\min\{d_{\lambda}, \sigma_{0}\}$ and $T_*=\frac{\sigma_{0}\mu}{2K_{*}\tilde{d}_{\lambda}}>0$. Now we distinguish two cases:\\
(1) $\tilde{\eta}(t)\in A_{3\mu/2, R}^{\lambda}\cap \Upsilon_{\kappa}\cap B_{r}^{\lambda}$ for $\forall\, t\in [0, T_*]$.\\
(2) $\tilde{\eta}(t_{0})\not\in A_{3\mu/2, R}^{\lambda}\cap \Upsilon_{\kappa}\cap B_{r}^{\lambda}$ for some $t_{0}\in [0, T_*]$.\\
If case (1) holds, we have $\Psi(\tilde{\eta}(t))\equiv 1$ and $\Vert  \Phi'_{\lambda, R}(\tilde{\eta}(t))\Vert \geq \tilde{d}_{\lambda}$ for all
$t\in [0, T_*]$. Thereby, by \eqref{small1},
\begin{align*}
\Phi_{\lambda, R}(\tilde{\eta}(T_*))=&\Phi_{\lambda, R}(u)+\int_{0}^{T_*}\frac{d}{ds}\Phi_{\lambda, R}(\tilde{\eta}(s))ds\\
\leq&c_{\Gamma}-\int_{0}^{T_*}\tilde{d}_{\lambda}ds\\
=&c_{\Gamma}-\tilde{d}_{\lambda}T_*\\
\leq&c_{\Gamma}-\frac{\sigma_{0}\mu}{2K_{*}}.
\end{align*}
If (2) holds, we must analyze the following situations:\\
(i) There exists $t_{2}\in [0, T_*]$ such that $\tilde{\eta}(t_{2})\not\in  \Upsilon_{\kappa}$, and thus, for $t_{1}=0$ it yields that
\begin{align*}
\Vert \tilde{\eta}(t_{2})-\tilde{\eta}(t_{1})\Vert_{\lambda, R}\geq \delta>\mu,
\end{align*}
because $\tilde{\eta}(t_{1})=u\in\Upsilon$.\\
(ii) There exists $t_{2}\in [0, T_*]$ such that $\tilde{\eta}(t_{2})\not\in  B_{r}^{\lambda}$, so that for $t_{1}=0$, we obtain
\begin{align*}
\Vert \tilde{\eta}(t_{2})-\tilde{\eta}(t_{1})\Vert_{\lambda, R}\geq r>\mu,
\end{align*}
because $\tilde{\eta}(t_{1})=u\in B_{r}^{\lambda}$.\\
(iii)  $\tilde{\eta}(t)\not\in \Upsilon_{\kappa}\cap B_{r}^{\lambda}$, and  there exist $0\leq t_{1}< t_{2}\leq T_*$ such that
$\tilde{\eta}(t)\in A_{3\mu/2, R}^{\lambda}\backslash A_{\mu, R}^{\lambda}$ for all $t\in [t_{1}, t_{2}]$ with
\begin{align*}
\vert \Phi_{\lambda, R}(\tilde{\eta}(t_{1}))-c_{\Gamma}\vert=\mu \,\, \text{and}\,\, \vert \Phi_{\lambda, R}(\tilde{\eta}(t_{2}))-c_{\Gamma}\vert=\frac{3\mu}{2}.
\end{align*}
From the definition of $K_{*}$,
\begin{align*}
 \Vert \tilde{\eta}(t_{2})-\tilde{\eta}(t)\Vert_{\lambda, R}&\geq \frac{1}{K_{*}}\Big\vert \Phi_{\lambda, R}(\tilde{\eta}(t_{2}))- \Phi_{\lambda, R}(\tilde{\eta}(t_{1}))\Big\vert\\
&\geq\frac{1}{K_{*}}\Big(\vert \Phi_{\lambda, R}(\tilde{\eta}(t_{2}))-c_{j_{0}}\vert-\vert \Phi_{\lambda, R}(\tilde{\eta}(t_{1}))-c_{j_{0}}\vert\Big)\\
&\geq\frac{1}{2K_{*}}\mu.
\end{align*}
By the mean value theorem and $t_{2}-t_{1}\geq \frac{1}{2K_{*}}\mu$, we find
\begin{align*}
\Phi_{\lambda, R}(\tilde{\eta}(T_*))=&\Phi_{\lambda, R}(u)+\int_{0}^{T_*}\frac{d}{ds}\Phi_{\lambda, R}(\tilde{\eta}(s))ds\\
\leq&\Phi_{\lambda, R}(u)-\int_{0}^{T_*}\Psi(\tilde{\eta}(s))\Vert \Phi'_{\lambda, R}(\tilde{\eta}(s))\Vert ds\\
\leq&c_{\Gamma}-\int_{t_{1}}^{t_{2}}\sigma_{0}ds\\
=&c_{\Gamma}-\sigma_{0}(t_{2}-t_{1})\\
\leq&c_{\Gamma}-\frac{\sigma_{0}\mu}{2K_{*}},
\end{align*}
which proves \eqref{small5}. \\
Fix $\widehat{\eta}(\textbf{t})=\eta(T_*, \gamma_{0}(\textbf{t}))$, we have that $\widehat{\eta}(\textbf{t})\in \Upsilon_{2\kappa}$, and so
$\widehat{\eta}(\textbf{t})|_{\Omega'_{j}}\neq 0$ for all $j\in\Gamma$. Thus, $\widehat{\eta}\in \Gamma_{*}$ and
\begin{align*}
b_{\lambda,R, \Gamma}\leq \underset{\textbf{s}\in [1/T^{2}, 1]^{l}}{\max} \Phi_{\lambda, R}(\widehat{\eta} (\textbf{s})) \leq\max\{m_{0}, c_{\Gamma}-\frac{\sigma_{0}\mu}{2K_{*}}\}< c_{\Gamma}.
\end{align*}
But, by Corollary \ref{coro1}, $b_{\lambda, R, \Gamma}\rightarrow c_{\Gamma}$ as $\lambda\rightarrow \infty$ uniformly hold in $R>0$ large, which is a contradiction.

Thus,  we can conclude that
$\Phi_{\lambda, R}$ has a critical point $u_{\lambda,R}\in A_{\mu}^{\lambda}$ for  $\lambda>0$ and $R>0$ large enough.

\end{proof}

\section{Proof of Theorem \ref{teorema}}
From Proposition \ref{prop72}, for $\mu\in (0, \mu^{*})$ and $\Lambda_{*}> 0$, there exists a
positive solution $u_{\lambda,R}$ for problem $(M_{\lambda, R})$ satisfying $u_{\lambda,R}\in A_{\mu,R}^{\lambda}\cap  \Phi_{\lambda, R}^{c_{\Gamma}}\cap B_{r+1}^{\lambda}$ for all $\lambda \geq \Lambda_{*}$ and $R \geq R^*$.

Now we will fix $\lambda \geq \Lambda_{*}$ and take a sequence $R_n \to +\infty$. Thereby, we have a solution $u_{\lambda,n}=u_{\lambda,R_n}$ for $(M_{\lambda, R_n})$ with
$$
u_{\lambda,n}\in A_{\mu, R_n}^{\lambda}\cap  \Phi_{\lambda, R_n}^{c_{\Gamma}}\cap B_{r+1}^{\lambda}, \quad \forall n \in \mathbb{N}.
$$
As $(u_{\lambda,n})$ is bounded in $H^{1}(\mathbb{R}^N)$,  we can assume that for some $u_\lambda \in H^{1}(\mathbb{R}^N)$
$$
\Phi_{\lambda,R_n}(u_{\lambda,n}) \to d \leq c_\Gamma,
$$
$$
u_{\lambda,n} \rightharpoonup u_\lambda \,\, \mbox{in} \,\, H^{1}(\mathbb{R}^N),
$$
$$
u_{\lambda,n} \to u_\lambda \quad \mbox{in} \,\, L^{q}_{\text{loc}}(\mathbb{R}^N) \,\,  \mbox{for any}\,\, q \in [1,2^*),
$$
and
$$
u_{\lambda,n}(x) \to u_\lambda(x) \quad \mbox{a.e.} \, x \in \mathbb{R}^N.
$$
Recall from Lemma \ref{le51} that
$$
0\leq u_{\lambda,n}(x) \leq a_{0} \quad \forall x \in \mathbb{R}^N \setminus \Omega_{\Gamma},
$$
we also have that
$$
0\leq u_{\lambda}(x) \leq a_{0} \quad \forall x \in \mathbb{R}^N \setminus \Omega_{\Gamma}.
$$

The next two lemmas play a fundamental role in the proof of Theorem \ref{teorema}. Since their proofs follow by similar arguments explored in Proposition \ref{prop1*}, we omit them.
\begin{lemma}\label{compactness} For any fixed $\zeta>0$, there exists $R>0$ such that
	$$
	\underset{n\rightarrow\infty}{\lim\sup}\int_{\mathbb{R}^{N}\backslash B_{R}(0)}(\vert \nabla u_{\lambda,n}\vert^{2}+(\lambda V(x)+1)\vert u_{\lambda,n}\vert^{2})dx\leq \zeta.
	$$
\end{lemma}

\begin{lemma} \label{compacidadeRR} $u_{\lambda,n}\to u_{\lambda}$ in $H^{1}(\mathbb{R}^N)$. Moreover,
	$$
	F_1(u_{\lambda,n}) \to F_1(u_{\lambda}) \quad \mbox{and} \quad  F'_1(u_{\lambda,n})u_{\lambda,n} \to F'_1(u_{\lambda})u_{\lambda} \quad \mbox{in} \quad L^{1}({\mathbb{R}}^N).
	$$
	As a consequence, setting the energy functional $\Phi_\lambda:E_{\lambda} \rightarrow (-\infty, +\infty]$ given by
	$$
	\Phi_{\lambda}(u)=\dis\frac{1}{2}\int \big (|\nabla u|^2+(\lambda V(x) +1)|u|^2\big)dx-\dis\frac{1}{2}\int u^2\log u^2dx,
	$$
the function $u_{\lambda}$ is a critical point of $\Phi_{\lambda}$ with
	\begin{align*}
u_\lambda \in  A_{\mu}^{\lambda}=\{u\in (\Upsilon_\infty)_{2\kappa}: \Phi_{\lambda, \mathbb{R}^N \backslash \Omega'_{\Gamma}}(u)\geq 0, \,\Vert u\Vert^{2}_{\mathbb{R}^N \backslash \Omega_{\Gamma}}\leq\mu, \,\,\vert I_{\lambda, j}(u)-c_{j}\vert\leq \mu, \,\, \forall j\in \Gamma\},
\end{align*}
where
$$
\Upsilon_\infty=\Big\{ u\in E_{\lambda}: \Vert u\Vert_{\lambda, \Omega'_{j}}>\frac{\tau}{2T}, \,\, \forall j\in\Gamma \Big\},
$$
and
$$
(\Upsilon_\infty)_{r}=\{u\in E_{\lambda}: \inf_{v\in \Upsilon_\infty}\Vert u-v\Vert_{\lambda, \Omega'_{j}}\leq r,\,\forall j\in\Gamma\}.
$$

Here, by a critical point we understand that $u_\lambda$ satisfies the inequality below
$$
\int \nabla u_\lambda \nabla (v-u_\lambda)\, dx + \int (\lambda V(x)+1)u_\lambda(v-u_\lambda)\,dx + \int F_1(v)\, dx - \int F_1(u_\lambda)\, dx \geq \int F'_2(u_\lambda)(v-u_\lambda)\,dx,
$$
for all $v \in E_\lambda$.	Hence, $u_\lambda$ satisfies the equality below
\begin{equation*}
	\displaystyle \int_{\mathbb{R}^N} (\nabla u_\lambda \nabla v +\lambda V(x)u_\lambda  v)dx=\displaystyle \int_{\mathbb{R}^N} u_\lambda v \log u_\lambda^2dx,\,\, \mbox{for all } v \in C^\infty_0(\mathbb{R}^{N}).
\end{equation*}
\end{lemma}

Now, we are ready to conclude the proof of Theorem \ref{teorema}.

\begin{proof} [Proof of Theorem \ref{teorema}]

Now, given $\lambda_n \to +\infty$ and $\mu_n\in (0, \mu^{*})$ with $\mu_n \to 0$, there exists a solution $u_n \in  A_{\mu_n}^{\lambda_n}$ of problem $(P_{\lambda_n})$. Therefore, $(u_n)$ is bounded in $H^{1}(\mathbb{R}^N)$ and satisfies:\\
(a)\, $\|\Phi'_{\lambda_{n}}(u_{\lambda_{n}})\|=0,\,\,\, \forall n\in \mathbb{N}$, \\
(b)\, $\Vert u_{\lambda_{n}}\Vert _{\lambda_{n},\mathbb{R}^N \backslash \Omega_{\Gamma} }\rightarrow 0$,\\
(c)\, $\Phi_{\lambda_{n}}(u_{n})\rightarrow d\leq c_{\Gamma}$.\\

Here,
$$
\|\Phi'_\lambda(u)\|=\sup\left\{\langle \Phi'_\lambda(u),z\rangle\,:\, z \in H_c^{1}(\mathbb{R}^N) \quad \mbox{and} \quad \|z\|_\lambda \leq 1 \right\}.
$$

Arguing as in Proposition \ref{prop1*}, there exists $u \in H^{1}(\mathbb{R}^N)$ such that
$u_{\lambda_{n}}\rightarrow u$ strongly in $H^{1}(\mathbb{R}^N)$, and  $u\equiv0$ in $\mathbb{R}^N \backslash \Omega_{\Gamma}$ and $u$ is a nontrivial solution of
$$
\left\{
\begin{array}{lc}
-\Delta u=u \log u^2, & \mbox{in} \quad \Omega_{\Gamma}, \\
%u(x)>0, &  \mbox{in} \quad \Omega_{\Gamma},  \\
u=0,  & \,\,\mbox{on} \quad \partial \Omega_{\Gamma},  \\
\end{array}
\right.
\eqno{(P_{\infty, \Gamma})}
$$
and so,
$$
I_{\Gamma}(u)\geq c_{\Gamma}.
$$
On the other hand, we also know that
$$
\Phi_{\lambda_{n}}(u_{\lambda_{n}})\rightarrow I_{\Gamma}(u),
$$
and so,
$$
I_{\Gamma}(u)=d\,\, \text{and}\,\,d\geq c_{\Gamma}.
$$
Since $d\leq c_{\Gamma}$, it yields that
$$
I_{\Gamma}(u)=c_{\Gamma},
$$
showing that $u$ is a least energy solution for $(P_{\infty, \Gamma})$. This completes the proof of the theorem.
\end{proof}

\section*{Acknowledgements}
The authors would like to thank the referee for many useful comments which clarify the paper.

\noindent \textsc{Claudianor O. Alves } \\
Unidade Acad\^{e}mica de Matem\'atica\\
Universidade Federal de Campina Grande \\
Campina Grande, PB, CEP:58429-900, Brazil \\
\texttt{coalves@mat.ufcg.edu.br} \\
\noindent and \\
\noindent \textsc{Chao Ji}(Corresponding author) \\
Department of Mathematics\\
East China University of Science and Technology \\
Shanghai 200237, PR China \\
\texttt{jichao@ecust.edu.cn}

\end{document}